\documentclass[]{amsart}
\usepackage{amsmath}
\usepackage{amssymb}
\usepackage{amsthm}
\usepackage{amsfonts}
\usepackage[abbrev]{amsrefs}
\usepackage{bm} 
\usepackage[dvipdfmx]{graphicx}
\usepackage{enumerate}
\usepackage{url}
\usepackage{epsfig}
\usepackage{color}
\usepackage{float}
\usepackage{setspace}
\usepackage{comment}
\usepackage{appendix}
\usepackage[hang,small,bf]{caption}
\usepackage[subrefformat=parens]{subcaption}
\theoremstyle{definition}

\newtheorem{thm}{Theorem}[section]

\newtheorem{lem}[thm]{Lemma}
\newtheorem{cor}[thm]{Corollary}
\newtheorem{prop}[thm]{Proposition}

\newtheorem*{thm*}{Theorem}
  
\theoremstyle{remark}
\newtheorem{rem}{Remark}[section]

\newtheorem*{acknowledgment}{Acknowledgments}

\numberwithin{equation}{section}\numberwithin{figure}{section}
\def\co{\colon\thinspace}

\makeatletter
\providecommand{\leftsquigarrow}{%
  \mathrel{\mathpalette\reflect@squig\relax}%
}
\newcommand{\reflect@squig}[2]{%
  \reflectbox{$\m@th#1\rightsquigarrow$}%
}
\makeatother

\title{On Matveev-Piergallini moves for branched spines}

\author{Kohei Muramatsu}%
\address{Mizuho Research \& Technologies, Ltd.,
5-16-6, Hakusan, Bunkyo-ku, Tokyo 112-0001, Japan}
\email{kohei.muramatsu@mizuho-rt.co.jp}

\author{Sakie Suzuki}%
\address{Department of Mathematical and Computing Science, School of Computing,
Institute of Science Tokyo,
2-12-1, Ookayama, Meguro-ku, Tokyo 152-8552, Japan}
\email{sakie@c.titech.ac.jp}

\author{Koki Taguchi}%
\address{Department of Mathematical and Computing Science, School of Computing,
Institute of Science Tokyo,
2-12-1, Ookayama, Meguro-ku, Tokyo 152-8552, Japan}

\begin{document}

\maketitle

\begin{abstract}
The Matveev-Piergallini (MP) moves on spines of $3$-manifolds are well-known for their correspondence to the Pachner $2$-$3$ moves in dual ideal triangulations. Benedetti and Petronio introduced combinatorial descriptions of closed $3$-manifolds and combed $3$-manifolds by using branched spines and their equivalence relations, which involve MP moves with 16 distinct patterns of branchings. In this paper, we demonstrate that these 16 MP moves on branched spines are derived from a primary MP move,  pure sliding moves, and their inverses. {Consequently, we obtain simpler combinatorial descriptions for closed $3$-manifolds and combed $3$-manifolds.}  Furthermore, we extend these results to framed $3$-manifolds and  spin $3$-manifolds.
These descriptions are advantageous, particularly when constructing and studying quantum invariants of links and $3$-manifolds.  In various constructions of quantum invariants using (ideal) triangulations, branching structures naturally arise to facilitate the assignment of non-symmetric algebraic objects to tetrahedra. 
In these frameworks, the primary MP move precisely corresponds to certain algebraic pentagon relations, such as the pentagon relation of the canonical element of a Heisenberg double, the Biedenharn-Elliott identity for quantum $6j$-symbols, or Schaeffer's identity for the Rogers dilogarithm and its non-commutative analog for Faddeev's quantum dilogarithm in quantum Teichm\"uller theory.
We expect our results to contribute to a better understanding of quantum invariants in the context of spines and ideal triangulations.
\end{abstract}

\tableofcontents

\section{Introduction}

\subsection{Primary MP move and pure sliding moves on normal o-graphs}
{In this paper, we assume that $3$-manifolds are connected, compact and oriented. 
We denote by $\mathcal{M}$ the set of closed $3$-manifolds up to the equivalence relation of orientation-preserving diffeomorphism.
A combing on a closed $3$-manifold is a homotopy class of nowhere-vanishing vector fields.
A combed $3$-manifold is a pair $(M,v)$, where $M$ is a closed $3$-manifold and $v$ is a combing on $M$, and we denote by $\mathcal{M}_{\text{comb}}$ the set of combed $3$-manifolds up to the natural action of orientation-preserving diffeomorphisms.}

A normal o-graph is a decorated $4$-valent graph which represents a branched spine of a  $3$-manifold, and {it can be described} by an oriented virtual link diagram. {The dual of a branched spine corresponds to a branched ideal triangulation, that is, an ideal triangulation in which the $4$ ideal vertices of each ideal tetrahedron are ordered so that when two tetrahedra share a common face the orders are compatible.} 

A closed normal o-graph is a normal o-graph which represents a branched spine of a closed $3$-manifold (see Section \ref{Closed normal o-graph}).
Benedetti and Petronio \cite{BP} established combinatorial {descriptions} of $\mathcal{M}$ and $\mathcal{M}_{\text{comb}}$  by using closed normal o-graphs and their equivalent relations involving 16 patterns of the Matveev-Piergallini (MP) moves {shown} in Figure \ref{fig:MP}.

\begin{figure}[H]
    \centering
    \includegraphics[scale=0.5]{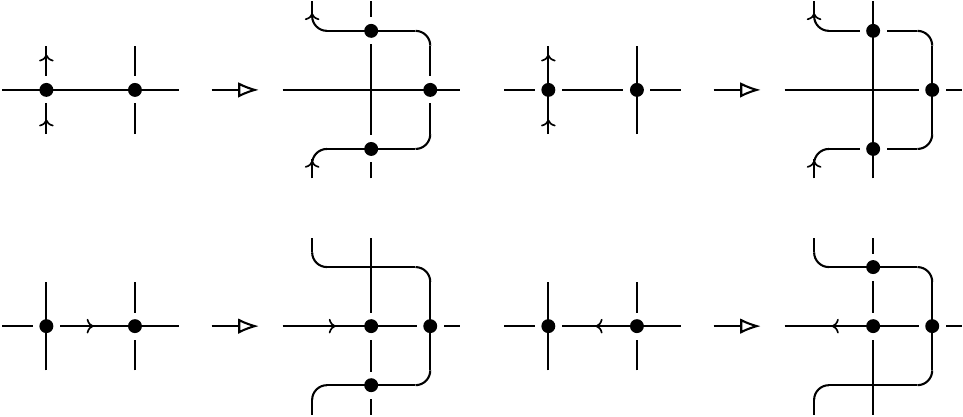}
    \caption{MP-moves. Orientation of each non-oriented edge is arbitrary if it matches before and after  move.}
    \label{fig:MP}
\end{figure}

{
We introduce the \textit{primary MP move} as in Figure \ref{fig:pMP}, which is equivalent to a specific MP move of type D2 (see {Figures} \ref{The MP moves of type D} and \ref{D2 and primary MP}). 
{It resembles the Reidemeister III move on link diagrams, with the key difference being that the diagram on the left-hand side includes two real crossings and one virtual crossing, while all three crossings in the diagram on the right-hand side are real.}  The primary MP move can be seen as a relation in the set of morphisms in the \textit{category of normal o-tangles} \cite{MST1}, similarly to the Reidemeister moves in the  category of tangles, generated as a monoidal category by fundamental diagrams (crossings, maxima and minima).}

\begin{figure}[H]
    \centering
    \includegraphics[scale=0.7]{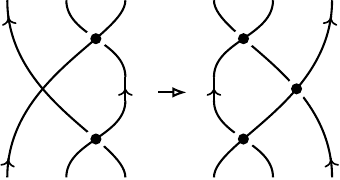}
    \caption{Primary MP move.}
    \label{fig:pMP}
\end{figure}
{Consider the boundary surfaces of the branched ideal  triangulations corresponding to the branched spines in the primary MP move. The bottom and top part of these surfaces are branched triangulated pentagons and the primary MP move induces a specific pentagon identity involving flip transformations, see Figure \ref{fig:Pentagon}.}
{This pentagon identity on branched triangulated surfaces plays a foundational role in the theory of quantum $6j$-symbols \cite{KR,TV} and quantum Teichm\"uller theory \cite{K, CF, AK}, as well as in the framework of quantum cluster algebras \cite{FG, FG1}. }

\begin{figure}[H]
    \centering
    \includegraphics[scale=0.45]{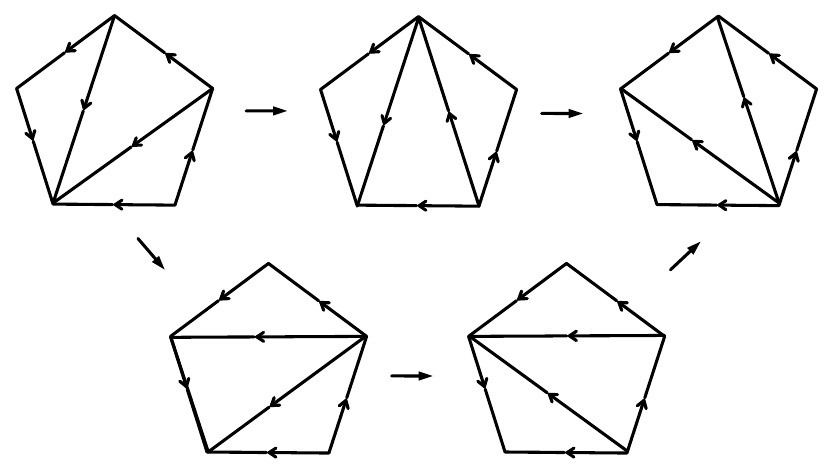}
    \caption{Pentagon identity of branched triangulations.}
    \label{fig:Pentagon}
\end{figure}
\noindent{The primary MP move was introduced in \cite{S} as a \textit{colored Pachner (2,3) move}, which was employed to study the universal quantum invariant of links \cite{Oh,Law} using ideal triangulations of link complements.
The MP move of type D2 {emerged} also in \cite{BS2} as a \textit{remarkable non ambiguous $b$-transit}\footnote{The positive crossing of a normal o-graph corresponds, in duality, to a negatively oriented branched tetrahedron. Thus, the primary MP move is,  in the terminology of \cite{BS2}, the remarkable non ambiguous $b$-transits involving only negatively $b$-oriented tetrahedra.}, utilized as a generator of a combinatorial description of the \textit{non ambiguous structure} of  $3$-manifolds.}

The \textit{pure sliding moves on branched spines} are defined in \cite{BP}. We translate this definition into the context of diagrams of normal o-graphs and define the \textit{pure sliding moves on normal o-graphs}, as shown in Figure~\ref{fig:RII}, which are nothing but the Reidemeister II moves on virtual link diagrams. These moves can only be performed when a specific global condition is satisfied on the normal o-graphs  (see Section \ref{Section;CPS}). This condition corresponds to a requirement on the $2$-cells of the branched spines (see Section~\ref{PS on BS}).

\begin{figure}[H]
    \centering
    \includegraphics[scale=0.8]{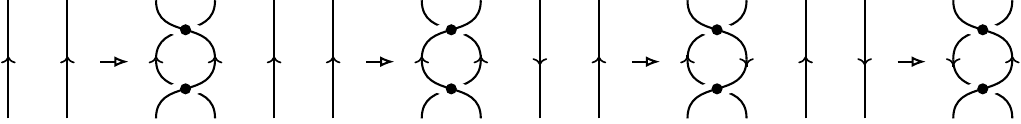}
    \caption{Pure sliding moves.}
    \label{fig:RII}
\end{figure}

\subsection{Main results}
In this paper we show that each MP move in Figure \ref{fig:MP} is derived  as  {a sequence of the primary  MP move}, the pure sliding moves, and  their inverses (Theorem \ref{th1}). 
Consequently, we have  alternative descriptions for $\mathcal{M}$ and $\mathcal{M}_{\text{comb}}$.

\begin{thm*}[Corollary  \ref{cocomb}]
There is a one-to-one correspondence between $\mathcal{M}_{\text{comb}}$ and the set of closed normal o-graphs, up to the primary MP move and pure sliding moves.
\end{thm*}

Corollary \ref{cocomb} {is} equivalent to the results in \cite[Theorem 2.3]{Ishii} by I. Ishii, where he {uses} flow spines to represent combed $3$-manifolds. His proof contains {a} geometrical argument involving non-singular flows and local sections in $3$-manifolds, while our proof reduces to combinatorial arguments on normal o-graphs (see Remark \ref{flow}).  
We can generalize Corollary \ref{cocomb}   to describe  $3$-manifolds (possibly with  non-empty boundary) with {concave vector fields \cite{BP2}}  by  using normal o-graphs  (not necessarily closed).

The CP (Combinatorial Pontryagin) move on closed normal o-graphs, introduced in \cite{BP}, is a move that relates different combings in the same 3-manifold.
The CP move allows us to eliminate the combing structure and obtain a combinatorial description of closed $3$-manifolds as follows.
\begin{thm*}[Corollary \ref{cotop}]
There is  a one-to-one correspondence between  $\mathcal{M}$  and the set of closed normal o-graphs, up to the primary MP move, pure sliding moves, and the CP move.
\end{thm*}

We can compare this result with a result by  F. Costantino \cite{C}, where he gave an alternative description of  $3$-manifolds using branched spines, employing bubble moves and bumping moves, which do not necessarily preserve the combings.

A framed $3$-manifold {is} a closed $3$-manifold with a framing, {defined as a homotopy class of a trivialization of the tangent bundle.}
Benedetti-Petronio \cite{BP} also gave combinatorial {descriptions} of framed $3$-manifolds and spin $3$-manifolds by using framed normal o-graphs, which are normal o-graphs having a $\mathbb{Z}_2$-weight on each edge (see Section \ref{sec:BP-diagrams}).
{In  \cite{MST2}, S. Mihalache, Y. Terashima and the second author gave a description of  framed $3$-manifolds (with vanishing first Betti number) using integral normal o-graphs, which, for closed $3$-manifolds, are integer lifts of framed normal o-graphs.}
In the present paper, we introduce pure sliding moves on integral normal o-graphs and show that these moves do not alter the framing of {the associated} $3$-manifolds (Proposition \ref{IPS move and framing}). Subsequently, we establish results similar to Theorem \ref{th1} for both integral normal o-graphs (Theorem \ref{thZ}) and framed normal o-graphs (Corollary \ref{thZ2}). As a consequence, we provide {simpler combinatorial descriptions for} framed $3$-manifolds (Corollary \ref{cofram}) and spin $3$-manifolds (Corollary \ref{cospin}).

\subsection{Primary MP move and five term relations arising from quantum invariants}
{In various constructions of quantum invariants using (ideal) triangulations, branching structures naturally arise to facilitate the assignment of non-symmetric algebraic objects to tetrahedra.
The primary MP move can be viewed as a topological realization of certain algebraic pentagon relations arising from such quantum invariants. Kashaev \cite{Kashaev} showed that a canonical element $S$ of the Heisenberg double of a finite dimensional Hopf algebra satisfies the \textit{pentagon relation} $S_{23}S_{12}=S_{12}S_{13}S_{23}$. The quantum invariant defined in \cite{S, MST1, MST2} by Mihalache, Terashima, and the second author is constructed based on associating a branched tetrahedron with the canonical element $S$ of a Heisenberg double. This construction is functorial and it maps a real positive crossing of normal o-graphs to the canonical element $S$, and the primary MP move to the pentagon relation $S_{23}S_{12}=S_{12}S_{13}S_{23}$. Through the theory of quantum $6j$-symbols, the primary MP move corresponds to the  \textit{Biedenharn-Elliott identity}. This framework yields in particular the reduced Turaev-Viro invariant \cite{BS2}, which, after symmetrization, recovers the Turaev-Viro invariant \cite{TV}.
Faddeev's quantum dilogarithm \cite{F}  plays a crucial role in the theory of quantum Teichm\"uller theory and of complex Chern-Simons theory. This quantum dilogarithm appears in the canonical element of the Heisenberg double of the Borel subalgebra of $U_{q}(sl_2)$.
The pentagon relation of Faddeev's quantum dilogarithm, which is a non-commutative analog of \textit{Schaeffer's identity} for the classical Rogers dilogarithm, corresponds to 
 the sequence of flip {transformations of branched surfaces} as {shown} in Figure \ref{fig:Pentagon}.
The quantum hyperbolic invariant \cite{BS0, BS1, BS2} is similarly defined based on the Heisenberg double, specifically using the 6j-symbols for the cyclic representation theory of the Borel subalgebra of {$U_{q}(sl_2)$ at a root of unity}.  In this context, the action of the canonical element on a cyclic irreducible representation is referred to as the \textit{basic matrix dilogarithm}, while the pentagon relation is termed the \textit{matrix Schaeffer's identity}. }

\subsection{Quantum invariant $Z$ for framed $3$-manifolds and cyclic moves}

The quantum invariant $Z$ constructed in \cite{MST2} using the Heisenberg double of a Hopf algebra $H$ is an invariant of framed $3$-manifolds with vanishing first Betti number.  When  $H$ is involutory, the invariant aligns with that in \cite{S}, reducing to {a} combed $3$-manifolds invariant. Additionally, if $H$ is unimodular and counimodular, the invariant simplifies to a topological $3$-manifolds invariant constructed in \cite{MST1}.
The proof of invariance  of $Z$ {requires} showing consistency under the 16 MP moves. By employing the primary MP move and the {pure sliding moves} instead of the 16 MP moves, the proof becomes more straightforward.

The obstruction of $Z$ to be topological invariant, or, in other words, the aspect related to framing, lies in its invariance under ``cyclic" moves. These moves involve a region surrounded by edges oriented cyclically, see Figure \ref{fig:cMP} for examples.\footnote{In \cite{MST2}, we classify MP moves of type A and type B, and type B corresponds the set of cyclic MP moves.} 
\begin{figure}[H]
    \centering
    \includegraphics[scale=0.8]{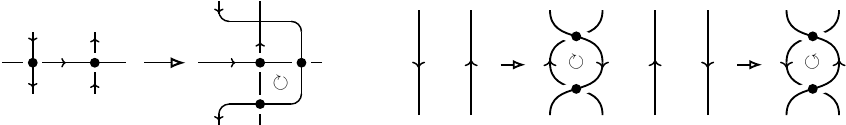} 
        \caption{Cyclic MP move and cyclic pure sliding moves.}
    \label{fig:cMP}
\end{figure}
\noindent When we use the primary MP move and the pure sliding moves, the obstruction is only the invariance under cyclic pure sliding moves. {More precisely, each non-cyclic MP move can be derived as {a sequence of the primary MP move} and non-cyclic pure sliding moves, whereas any decomposition of a cyclic MP move as a sequence of moves must include at least one cyclic pure sliding move.}
 {Thus,} we can say that the obstruction essentially resides in the cyclic pure sliding moves, {see} Appendix \ref{cyclic} for details.

\subsection{Moves on normal o-graphs and algebraic equations of quantum invariants}
One of the most interesting challenges in the study of quantum invariants is understanding the relationship between three-dimensional topology and the algebraic properties of these invariants.
In response to Atiyah's quest for an intrinsically three-dimensional definition of the Jones polynomial, Witten \cite{W} used Chern-Simons theory and three-dimensional topological quantum field theory, paving the way for a possible topological interpretation of the Jones polynomial. This approach also led him to propose new types of $3$-manifold invariants. Witten employed a path integral, a concept not rigorously defined in mathematics, yet his theory provided substantial insights. 
However,  a comprehensive understanding within a mathematically rigorous framework remains elusive, and the present paper might take a step towards addressing this gap {by establishing a three-dimensional framework for studying generalizations of the Jones polynomial. More precisely,}
the universal quantum invariant of links \cite{Oh,Law} {recovers} the Reshetkhin-Turaev invariants \cite{RT1} of links, {in particular, the colored Jones polynomials}. 
From the universal quantum invariant, we can {also} construct the Witten-Reshetkhin-Turaev invariant \cite{RT2} of $3$-manifolds.  By extending the results in \cite{S,  MST2}, in  \cite{S2}, the second author reconstructed the universal quantum invariant  for the Drinfeld double using integral normal o-graphs and  the invariant $Z$. In this sense, the invariant $Z$ might incorporate the Reshetkhin-Turaev invariant of links (as invariants of link complements with framings) and a framing refinement of the Witten-Reshetkhin-Turaev invariant.
{Our} {description} of $3$-manifolds using the primary MP move and the pure sliding moves aligns with the invariant $Z$, where the correspondence between topological phenomena and algebraic structure through the invariant is straightforward and comprehensible: {a} crossing of a normal o-graph (or an ideal tetrahedron in the dual ideal triangulation) corresponds to the canonical element; the primary MP move (or the Pachner $2$-$3$ move) embodies the pentagon relation of the canonical element, while the pure sliding moves (or the Pachner $0$-$2$ moves) reflect its invertibility. {In this harmonious correlation, understanding the moves of normal o-graphs implies understanding the corresponding algebraic equations of the invariants.} We {expect} our results to contribute to a better understanding of quantum invariants in the context of spines and ideal triangulations.

\subsection{Organization of this paper}
The paper is organized as follows. 
We start by discussing the {description} of combed $3$-manifolds and closed $3$-manifolds in Section \ref{Section_normal_o_graph}.
Here, we revisit the definition of branched spines in Section \ref{Branched spine} and explain a one-to-one correspondence between branched spines and normal o-graphs in Section \ref{Normal o-graph}. In Section \ref{Closed normal o-graph}, we recall the combinatorial {description} of combed $3$-manifolds and closed $3$-manifolds using closed normal o-graphs. 
We {then proceed with} the study of the pure sliding moves, initially we recall the original definition from \cite{BP} on branched spines in Section \ref{PS on BS}, and subsequently we introduce its combinatorial realization on normal o-graphs in Section \ref{Section;CPS}.
Our main results, outlined in Section \ref{Section;Main_results}, are detailed, with corresponding proofs provided in Section \ref{Section;proof}. 
The discussion extends to local pure sliding moves in Section \ref{localPS}.
{We then focus on framed $3$-manifolds and integral normal o-graphs in Section  \ref{Section;frame}.}
Here, {in Sections \ref {subsec: framing} and \ref{sec:BP-diagrams}, we recall the description of framed $3$-manifolds using branched spines and integral normal o-graphs.}
We then recall the combinatorial {description} of framed $3$-manifolds using framed normal o-graphs in Section \ref{Section;framed_o-graphs}.
In subsequent sections, including Sections \ref{Eulercochain} and  \ref{Section;IPS}, we introduce integral pure sliding moves and investigate their impact on framings. 
Our main results regarding framed $3$-manifolds are presented in Section \ref{Section;Main_results_framing}, with a further refinement provided in Section \ref{integrallocalPS}.
In Section \ref{Section;spin}, we give the results for spin normal o-graphs and spin $3$-manifolds.
The appendices provide further analyses, including an investigation into the symmetries of the MP moves in Appendix \ref{Section;symmetry}, and a study of cyclic moves in Appendix~\ref{cyclic}.

\begin{acknowledgment}
{We would like to thank Yuya Koda and Serban Matei Mihalache for valuable discussions.
We also thank Anderson Vera for his helpful comments on earlier drafts.
This work was partially supported by JSPS KAKENHI Grant Number JP19K14523.}
\end{acknowledgment}

 \section{Normal o-graphs and combed $3$-manifolds} \label{Section_normal_o_graph}
In Sections \ref{Branched spine}--\ref{Closed normal o-graph}, we follow the notation in \cite{MST2}.

\subsection{Branched spines and associated vector fields}\label{Branched spine}

\label{subsec:Branched spine}
For an introduction to standard and branched spines, see e.g. \cites{BP,Mat}\footnote{In \cite{Mat}, standard spine is called special spine.}.
A 2-dimensional compact polyhedron $P$ is called \textit{{simple}} if the neighborhood of each point $x\in P$ is homeomorphic to one of the pictures in Figure \ref{fig:nbh of simple spine}, where from the left in the picture the point $x$ is called a \textit{{non-singular point}}, a \textit{{triple point}}, and a \textit{{true vertex}}, respectively.

\begin{figure}[H]
    \centering
    \includegraphics[scale=0.9]{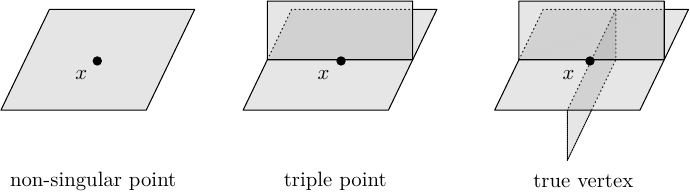}
    \caption{Neighborhood of point in $P$.}
    \label{fig:nbh of simple spine}
\end{figure} 

\noindent For a simple polyhedron $P$, set
\begin{align*}
  V(P)&=\{x\in P\mid \text{$x$ is a true vertex}\},\\
  S(P)&=\{x\in P\mid \text{$x$ is a true vertex or a triple point}\},\\
  D(P)&=\{x\in P\mid \text{$x$ is a non-singular point}\}.
\end{align*}

A simple polyhedron $P$ is called \textit{{standard}} if connected components of $S(P)\backslash V(P)$ and $D(P)$ are 1-cells and 2-cells, respectively.
Let $M$ be a $3$-manifold with non-empty boundary.
A standard polyhedron $P$ embedded in $\text{Int}M$ is called a \textit{{standard spine}} of $M$ if $M$ collapses to $P$.
It is known that every compact $3$-manifold with non-empty boundary admits a standard spine.
A standard spine $P$ of $M$ determines the homeomorphism class of $M$, i.e., if $P^{\prime}$ is a standard spine of $M^{\prime}$ which is homeomorphic to $P$, then $M^{\prime}$ is homeomorphic to $M$. 

An \textit{{oriented standard polyhedron}} is a standard polyhedron which is a spine of an oriented $3$-manifold \cite{BP}*{Proposition 2.1.2}.
An \textit{{oriented branching}} on an oriented standard polyhedron $P$ is an orientation on its $2$-cells such that on each $1$-cell, the orientations induced from the $2$-cells attached to it are not compatible, i.e., there are locally three $2$-cells which are attached to a $1$-cell $e$ and one of the three induced orientations on $e$ is opposite to the other two (cf. \cite{BP}*{Corollary 3.1.7}). 
We can visualize a branching structure on $P$ as a smoothing of $P$  as shown in Figure \ref{fig:branching}, where the ``branching" starts from the region which induces inverse orientation on the $1$-cell relative to the others.
Here, each 1-cell has a canonical orientation as the two compatible orientations induced by the $2$-cells attached to it. 
Up to orientation-preserving homeomorphism, there exist two possibilities for the branching structure near the 0-cell, the \textit{{type $+$}} and the \textit{{type $-$}},  which are shown in Figure \ref{fig:branched vertex}. 

\begin{figure}[H]
  \begin{minipage}[b]{0.38\linewidth}
    \centering
    \includegraphics[keepaspectratio, scale=0.3]{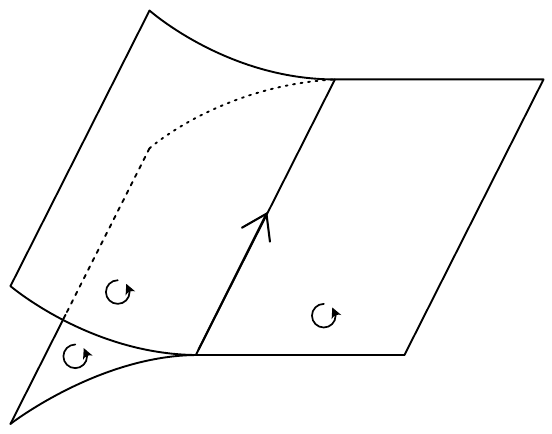}
    \caption{Branching.}
    \label{fig:branching}
  \end{minipage}
  \begin{minipage}[b]{0.6\linewidth}
    \centering
    \includegraphics[keepaspectratio, scale=0.35]{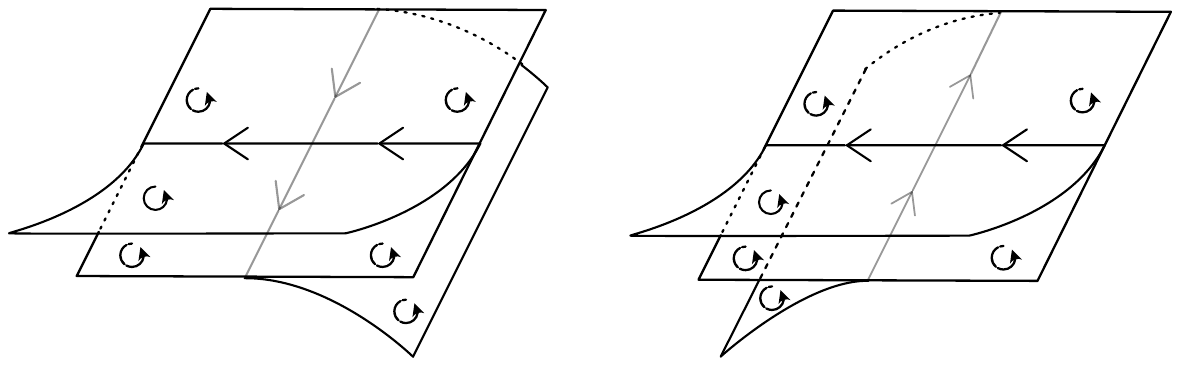}
    
    \begin{picture}(0,0)
    \put(-90,0){type $+$}
    \put(55,0){type $-$}
    \end{picture}
    \caption{Local branching near  true vertex.}\label{fig:branched vertex}
  \end{minipage}
\end{figure}

By abusing the terminology, we refer to an oriented standard polyhedron with an oriented branching as a \textit{{branched polyhedron}}.
Let $P$ be a branched polyhedron and $M(P)$ the $3$-manifold obtained by thickening $P$.
Then $P$ defines a unique nowhere-vanishing vector field $v(P)$ on $M(P)$, perpendicular to $P$ and following a right-hand screw direction as depicted in Figure \ref{fig:combin of branching}.
\begin{figure}[H]

  \centering
   \includegraphics[scale=0.9]{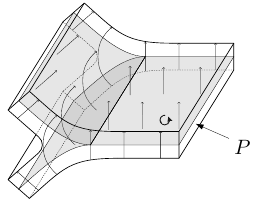}
   \caption{Vector field on $M(P)$.}
   \label{fig:combin of branching}

\end{figure}

\subsection{Normal o-graphs and branched polyhedrons}\label{Normal o-graph}
A \textit{normal o-graph} is an oriented virtual link diagram up to planer isotopy and Reidemeister type moves defined in Figure \ref{fig:RM}. 

\begin{figure}[H]
    \centering
    
    \includegraphics{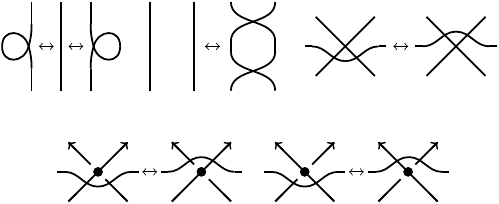}
    \caption{Reidemeister type moves.}
    \label{fig:RM}
\end{figure}

A normal o-graph  represents a branched polyhedron in a combinatorial manner.
For a branched polyhedron $P$, we replace each true vertex with a real crossing in $\mathbb{R}^2$; 
 we replace each vertex of type $+$ or $-$ to a positive or negative crossing, respectively, see Figure \ref{fig:BStoBP}.
\begin{figure}[H]
    \centering
    \includegraphics[scale=0.4]{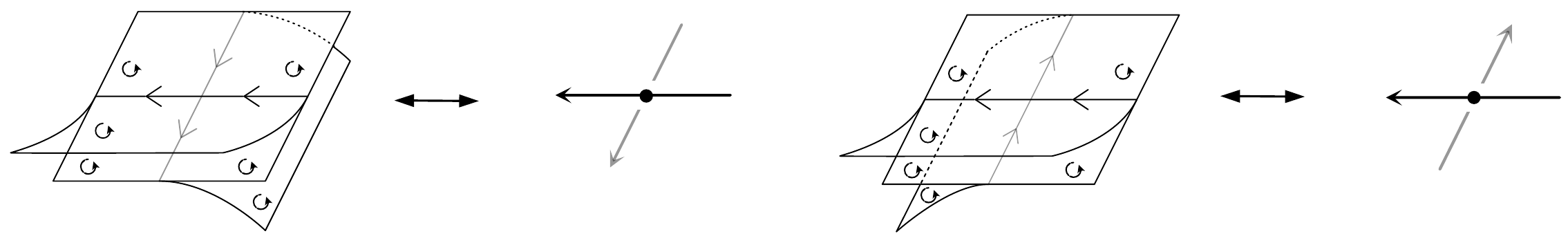}
    \caption{Correspondence between branched polyhedron and virtual link diagram near vertex of Type + (left) and type - (right).}
    \label{fig:BStoBP}
\end{figure}

\noindent Then we obtain a normal o-graph by connecting them according to how the true vertices are connected.  Note that this construction is unique up to planer isotopy and the Reidemeister type moves.

Conversely, we can construct in a natural way from an oriented virtual link diagram a homeomorphism class of a branched polyhedron. As a result, homeomorphism classes of branched polyhedra correspond one-to-one with normal o-graphs. In the following discussion, we occasionally treat branched polyhedra and normal o-graphs interchangeably.

\subsection{Closed normal o-graphs and combed $3$-manifolds}\label{Closed normal o-graph}
We recall from \cite{BP} the presentations of combed $3$-manifolds and closed $3$-manifolds.

Let $M$ be a closed  $3$-manifold.
A \textit{{combing}} $v$ of $M$ is a nowhere-vanishing vector filed over $M$. A combed $3$-manifold is a pair $(M,v)$ of a closed oriented $3$-manifold $M$ and its combing $v$.
Two combed $3$-manifolds $(M,v)$ and $(M^{\prime},v^{\prime})$ are \textit{{equivalent}} if there exists an orientation-preserving diffeomorphism $h\co M\to M^{\prime}$ such that $h_*v$ is homotopic through combings to $v^{\prime}$. We denote by $\mathcal{M}_{\text{comb}}$ the set of equivalence classes of combed $3$-manifolds.

Note that for a branched polyhedron $P$, the vector field $v(P)$ of $M(P)$ induces a stratification  $\partial M=\partial_{\text{out}} M\cup \partial_{\text{tan}} M\cup\partial_{\text{in}} M$ on the boundary of $M(P)$, where 
\begin{align*}
    \partial_{\text{out}} M &= \{x\in\partial M\,|\,\text{the field $v$ at $x$ points outward}\},\\
    \partial_{\text{tan}} M &= \{x\in\partial M\,|\,\text{the field $v$ at $x$ is tangent to }\partial M\},\\
    \partial_{\text{in}} M &= \{x\in\partial M\,|\,\text{the field $v$ at $x$ points inward}\},
\end{align*}
as shown in Figure \ref{fig:black white decomposition}. 
A \textit{{closed branched polyhedron}} $P$ is a branched polyhedron such that $\partial M(P)$ is diffeomorphic to the 2 dimensional sphere $S^2=\{(x,y,z)\in\mathbb{R}^3\,|\,x^2+y^2+z^2=1\}$ with the stratification given by $\partial_{\text{out}} S^2=\{z < 0\}$, $\partial_{\text{tan}} S^2=\{z = 0\}$, $\partial_{\text{in}} S^2=\{z > 0\}$.
For a closed branched polyhedron $P$, the associated vector field $v(P)$ on $M(P)$ extends to the combing $\widehat v(P)$ of closure $\widehat M(P)$ of $M(P)$ by capping the $S^2$ boundary with a ball having the trivial flow $(B^3,\frac{\partial}{\partial z})$. 

\begin{figure}[H]

    \centering
    \includegraphics[scale=1]{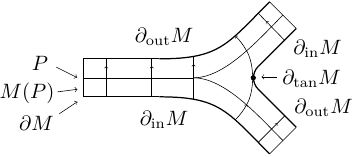}
    \caption{Stratification of boundary of $M(P)$.}
    \label{fig:black white decomposition}

\end{figure}

A \textit{closed normal o-graph} is a normal o-graph which represents a closed branched polyhedron. 
We denote by $\mathcal{G}$ the set of closed normal o-graphs. 
See Figure \ref{lens_space} for examples of closed normal o-graphs representing the Lens spaces $L(p,1)$ for $p\geq 0$, where the normal o-graph contains $p$ true vertices.
\begin{figure}[H]
  \centering
  \includegraphics{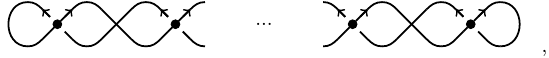}
  \caption{Normal o-graph representing Lens space $L(p,1)$. }
  \label{lens_space}
\end{figure}
The above construction defines a map
\begin{align*}
\Phi_{\rm comb} \co \mathcal{G} \to \mathcal{M}_{\rm comb}, \ \ \ \Gamma \mapsto (\widehat M(\Gamma), \widehat v(\Gamma)),
\end{align*}
which is in fact surjective \cite[Proposition 5.2.3]{BP}.

\begin{prop}[{Benedetti-Petronio \cite[Theorem 1.4.1]{BP}}]\label{BPcomb}
The equivalence relation defined by $\Phi_{\rm comb}$ is generated by the \textit{$0$-$2$ move} defined in Figure \ref{fig:02} and  the MP moves defined in Figure \ref{fig:MP}.  
\end{prop}

We denote by $\mathcal{M}$ the set of closed $3$-manifolds up to orientation-preserving diffeomorphism.
We have a surjective map
\begin{align*}
\Phi \co \mathcal{G} \to \mathcal{M}, \ \ \ \Gamma  \mapsto \widehat M(\Gamma).
\end{align*}

\begin{prop}[{Benedetti-Petronio \cite[Theorem 1.4.2]{BP}}]\label{BPclosed}

The equivalence relation defined by $\Phi$ is generated by the $0$-$2$ move, the MP moves, and the \textit{CP move} defined in \cite[Figure 1.6]{BP}\footnote{We will not show the figure of the CP move because we do not use it in this paper.}. 
\end{prop}

\begin{figure}[H]
    \centering
    \includegraphics[scale=0.8]{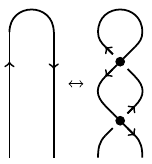}
    \caption{$0$-$2$ move. }
    \label{fig:02}
\end{figure}


\subsection{Pure sliding moves on branched polyhedrons}\label{PS on BS}
The \textit{pure sliding  moves} on branched  polyhedrons are introduced in \cite[Figure 4.7]{BP}, which is defined as in Figure \ref{fig:PS move}.\footnote{Note that  pure sliding moves I and II have the same configuration after rotation, but the orientations of $2$-cells do not match. In \cite[Figure 4.7]{BP}, both pure sliding moves I and II are depicted with a single picture, whereas we distinguish between them  by considering the orientations of $2$-cells.}
\begin{figure}[H]
        \centering
    \includegraphics[scale=0.9]{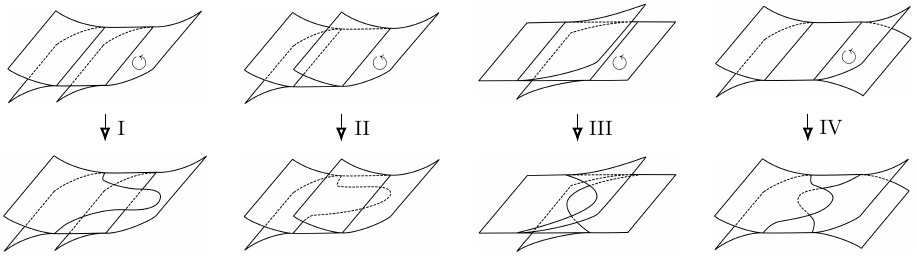}
    \caption{Pure sliding moves I--IV on branched polyhedrons.}
    \label{fig:PS move}

\end{figure}
It is not difficult to check that the $0$-$2$ move is a pure sliding move. 
We should emphasize that we cannot represent the pure sliding moves solely using a diagram of two edges of a normal o-graph, because such local diagram does not cover the global information about the arrangement of $2$-cells,  see Figure \ref{fig:PSng} for example. 

\begin{figure}[H]
        \centering
    \includegraphics[scale=0.9]{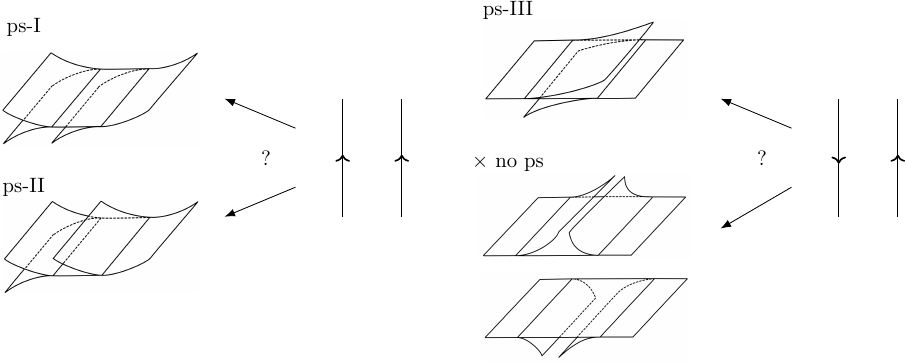}
    \caption{Possible configurations of branched polyhedra arising from local edges of normal o-graphs.}
    \label{fig:PSng}
\end{figure}

In the next section we will give a combinatorial method for performing the pure sliding moves on normal o-graphs.

\subsection{Pure sliding moves on normal o-graphs}\label{Section;CPS}
To implement the pure sliding moves on normal o-graphs, it is necessary to consider  $2$-cells within a branched polyhedron. Let $\Gamma$ be a normal o-graph and $P$ the corresponding branched polyhedron. Benedetti and Petronio (\cite[Figure 1.3]{BP}) introduced a procedure (Method A) to associate a union of circuits with a normal o-graph $\Gamma$, as illustrated in Figure \ref{fig:Method A}. These circuits represent the boundaries of $2$-cells and their attaching maps to the $1$-skeleton $S(P)$.
\begin{figure}[H]
    \centering
    \includegraphics[scale=0.8]{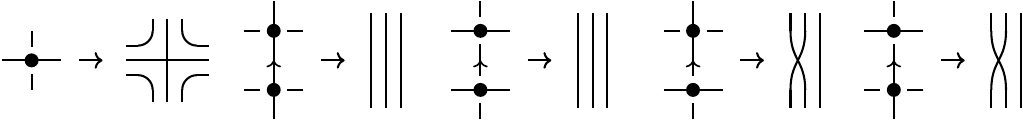}
    \caption{Method A.}
    \label{fig:Method A}
\end{figure}
We slightly modify their method as follows: we replace the neighborhood of each true vertex, as shown in Figure \ref{fig:Method B}, and connect them along edges using parallel curves (Method B). 
It is important to note that Method A and Method B convey the same information regarding $2$-cells and their attaching maps.
\begin{figure}[H]
    \centering
    \includegraphics[scale=0.9]{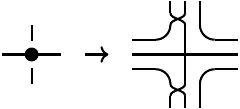}
    \caption{Method B.}
    \label{fig:Method B}
\end{figure}

An  advantage to use Method B is that the procedure is functorial\footnote{We can describe it as a functor between certain categories.}, where we replace neighborhoods of true vertices and edges independently,  while with Method A the circuit diagram around an edge depends on the two true vertices on the boundary of the edge. 
Another advantage, which is essential in the following argument, is that the following orders are consistent; 
\begin{itemize}
\item the order of three lines in the circuit diagram around an edge: we count from  left to right  seeing the edge of normal o-graph going upwards, and
\item the order of three (local) $2$-cells at the triple line: we count by the right-hand screw order which ends the $2$-cell starting the branching,
\end{itemize}
see Figure \ref{fig: corresponding 2-call and circuis}.
 Note that if there appears a crossing in the circuit diagram around an edge, as in Method A, then the order is broken at some point.      

\begin{figure}[H]
    \centering
    \includegraphics[scale=0.7]{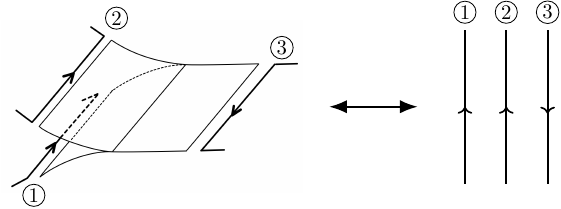}
    \caption{Correspondence between configuration of $2$-cells around triple line and lines in circuit diagram.}
    \label{fig: corresponding 2-call and circuis}
\end{figure}

Method B of circuit diagrams enables us to define the pure sliding moves combinatorially on normal o-graphs as follows.

A \textit{pure sliding move} on normal o-graph is defined as in Figure \ref{the pure sliding move}, where the normal o-graph before we perform each move ps-I -- ps-IV should respectively satisfy the following condition: on the circuit diagram by Method B, around the two edge where we perform the move, we have
\begin{itemize}
\item[\rm (PS-I)] the third line on the left is connected to the second line on the right,
\item[\rm (PS-II)] the third line on the left is connected to the first line on the right,
\item[\rm (PS-III)] the first line on the left is connected to the second line on the right,
\item[\rm (PS-IV)] the third line on the left is connected to the third line on the right,
\end{itemize}
see Figure \ref{the pure sliding move condi}.

\begin{figure}[H]
    \centering
    \includegraphics[scale=0.8]{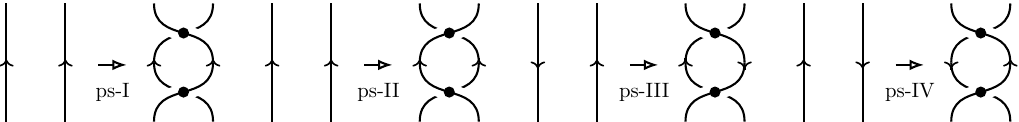}
    \caption{Pure sliding move ps-I -- ps-IV.}
    \label{the pure sliding move}
\end{figure}

\begin{figure}[H]
    \centering
    \includegraphics[scale=0.6]{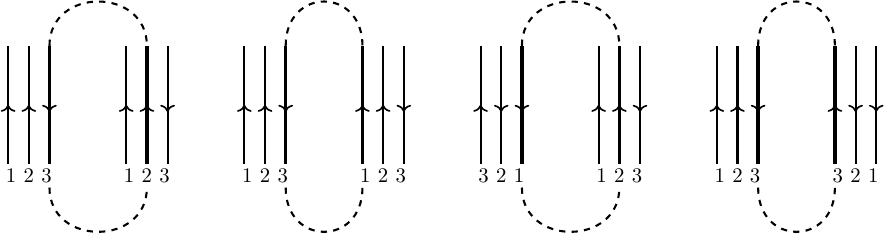}
    \caption{Conditions (PS-I)-- (PS-IV) for ps-I -- ps-IV, respectively.}
    \label{the pure sliding move condi}
\end{figure}

One can verify that each condition (PS-I)--(PS-IV) corresponds to the connection of the $2$-cells involved in the pure sliding moves I--IV on the branched spine. For example,  when we count the parts of the $2$-cells in branched polyhedron involved in the pure sliding move I as in the right in Figure \ref{PS1}, then the $\Delta_3$ is attached to the third strand on the left on the circuit diagram, and to the second on the right. 

\begin{figure}[H]
    \centering
    \includegraphics[scale=0.6]{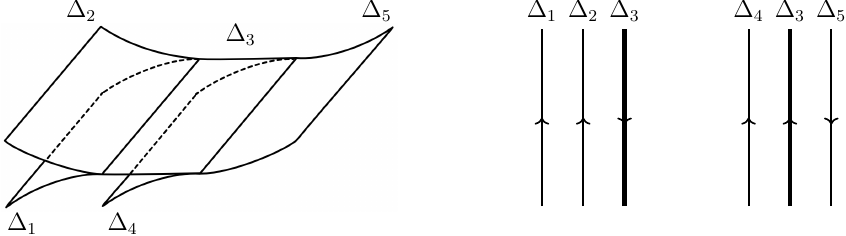}
    \caption{Configuration of branched polyhedron on which we perform pure sliding move I (left) and condition for performing ps-I on normal o-graph (right).}
    \label{PS1}
\end{figure}

\begin{rem}
The pure sliding moves on normal o-graphs are local in the sense that the normal o-graph does not change the out of the sliding region, but is not local in the sense that we need global information to check the conditions (PS-I)--(PS-IV).  In Section \ref{localPS}
we define local pure sliding moves on normal o-graphs, which are local in both of the above sense.
\end{rem}
 
\subsection{Main results for closed $3$-manifolds and combed $3$-manifolds}\label{Section;Main_results}
We state the main results for closed $3$-manifolds and combed $3$-manifolds. The proof of Theorem \ref{th1} will be presented in the next section. 


\begin{thm}\label{th1}
Each of the 16 MP moves in Figure \ref{fig:MP} is derived as a sequence of the primary MP move in Figure \ref{fig:pMP}, the pure sliding moves in Figure \ref{the pure sliding move}, and their inverses.
\end{thm}
Consequently we can replace the 16 MP moves in the equivalence relation for $\Phi$ and $\Phi_{\rm comb}$ in Propositions \ref{BPcomb} and \ref{BPclosed} by the primary MP move and the pure sliding moves.
\begin{cor}\label{cocomb}
The equivalence relation defined by the surjective map
$\Phi_{\rm comb} \co \mathcal{G} \to \mathcal{M}_{\rm comb}$
is generated by the primary MP move and the pure sliding moves. 
\end{cor}
\begin{cor}\label{cotop}
The equivalence relation defined by the surjective map
$\Phi \co \mathcal{G} \to \mathcal{M}$
is generated by the primary MP move, the pure sliding moves and the CP move. 
\end{cor}

\begin{rem}\label{flow}
Corollary \ref{cocomb}  are equivalent to  \cite[Theorem 2.3]{Ishii} by Ishii, where he used flow spines for closed $3$-manifolds. His proof contains methods involving non-singular flows and local sections in $3$-manifolds, while our proof contains only combinatorial argument on normal o-graphs (after  admitting the results in \cite{BP}). The first regular moves $R_1$ in \cite{Ishii} corresponds to \textit{the primary MP move negative}, that is, the primary MP move with each crossing replaced by negative crossings, and  the inverse of the second regular moves $R_2$-$(x), x=1,2,3,4,$ correspond to the pure sliding moves. The  ``regular equivalence" in \cite{Ishii}  corresponds to the equivalence relation on closed normal o-graphs generated by the primary MP move negative and the four moves in Figure \ref{the pure sliding move} without condition (PS-I)--(PS-IV), and ``strongly regular equivalence" corresponds equivalence relation generated by the primary MP move negative and moves in Figure \ref{the pure sliding move} requiring involved normal o-graphs to be closed.  This condition is combinatorially nothing but the conditions (PS-I)--(PS-IV). Note also that Theorem \ref{th1} works for not only closed normal o-graphs but also arbitrary normal o-graphs representing $3$-manifolds with boundary.

\end{rem}
\subsection{Proof of Theorem \ref{th1}}\label{Section;proof}
We prove Theorem \ref{th1}.
We classify the 16 MP moves into four types A, B, C, D depending on their shape, and then label $1,2,3,4$ in each type depending on the orientation of the two edges with arbitrary orientation, see Figure \ref{The MP moves of type A}-\ref{The MP moves of type D}.

\begin{figure}[H]
    \centering
    \includegraphics[scale=0.8]{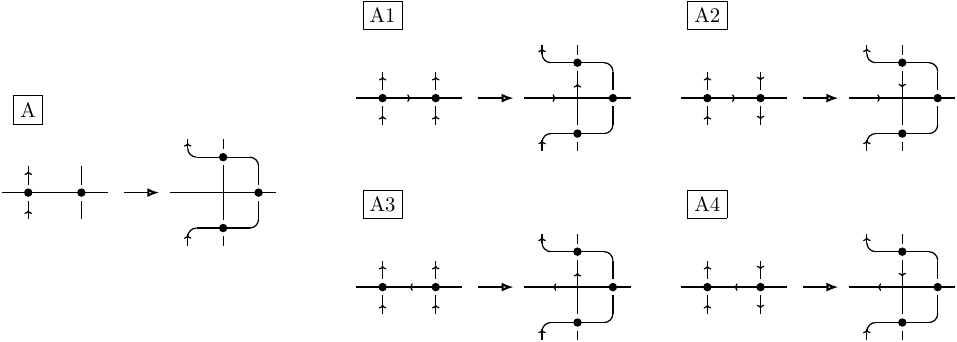}
    \caption{MP moves of type A.}
    \label{The MP moves of type A}
\end{figure}

\begin{figure}[H]
    \centering
    \includegraphics[scale=0.8]{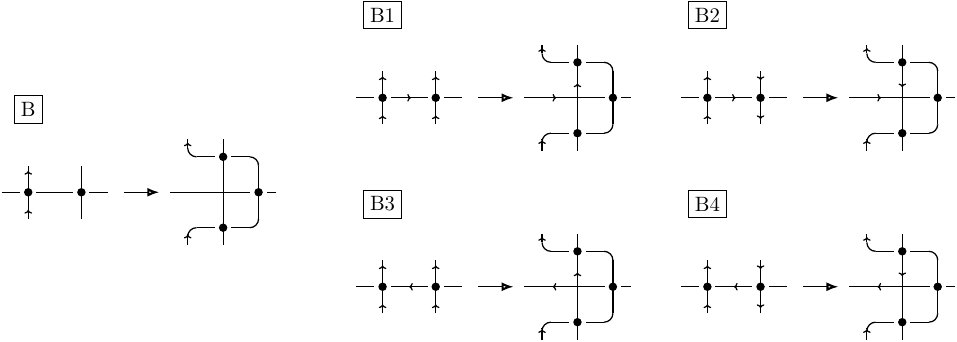}
    \caption{MP moves of type B.}
    \label{The MP moves of type B}
\end{figure}
\begin{figure}[H]
    \centering
    \includegraphics[scale=0.8]{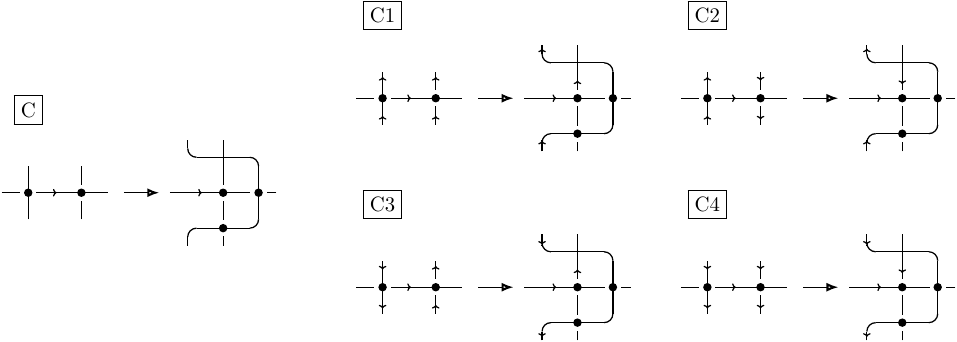}
    \caption{MP moves of type C.}
    \label{The MP moves of type C}
\end{figure}
\begin{figure}[H]
    \centering
    \includegraphics[scale=0.8]{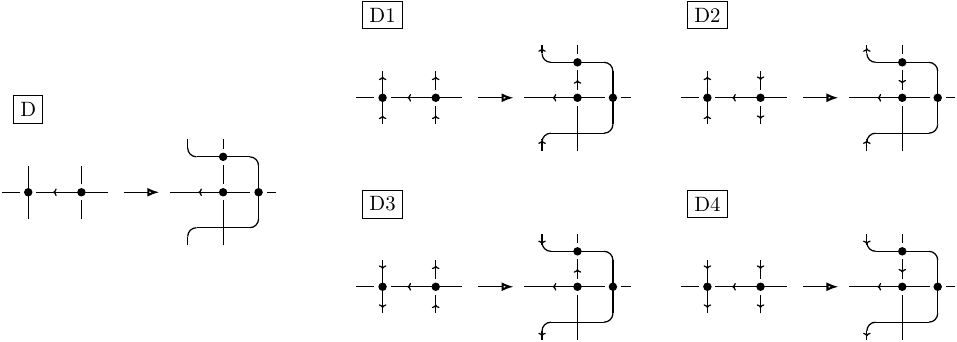}
    \caption{MP moves of type D.}
    \label{The MP moves of type D}
\end{figure}

The primary MP move is equivalent to $\rm D2$ as in the following  figure. 
\begin{figure}[H]
    \centering
    \includegraphics[scale=0.7]{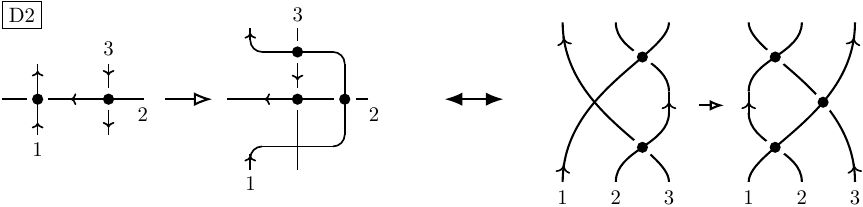}
    \caption{Correspondence between  D2 and primary MP move.}
    \label{D2 and primary MP}
\end{figure}
\noindent Thus, we will prove that each MP move is derived as a sequence of $\rm D2$,  the pure sliding moves, and their inverses. 

See the following figure.

\begin{figure}[H]
    \centering
    \includegraphics[scale=0.7]{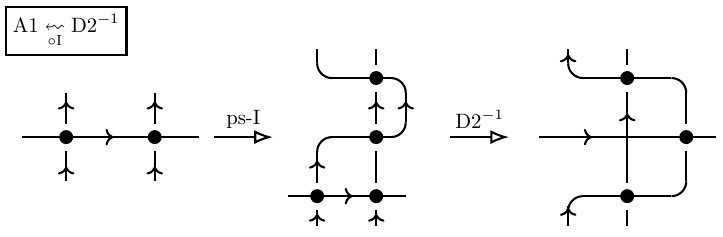}
    \caption{A1 is sequence of D$2^{-1}$ and ps-I.}
    \label{D2toA2}
\end{figure}
\noindent The left picture and the right picture are normal o-graphs before and after performing $\rm A1$, and thus $\rm A1$ is derived as a sequence of the ps-I and $\rm D2^{-1}$ in this order.
In this case, we denote by $\rm A1 \underset{\circ \mathrm{I}}{\leftsquigarrow}  D2^{-1}$ standing for ``$\rm A1= (D2^{-1}) \circ (ps$-I$)$" regarding moves as maps.  
Here, we can check the condition (PS-I)  as in Figure \ref{A1circ} for performing the first ps-I.

\begin{figure}[H]
    \centering
    \includegraphics[scale=0.7]{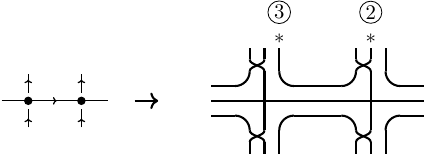}
    \caption{Confirming condition (PS-I).}
    \label{A1circ}
\end{figure}

When $\rm A1 \underset{\circ \mathrm{I}}{\leftsquigarrow} D2^{-1}$, we also have $\rm  A1^{-1}\underset{\mathrm{I}^{-1} \circ }{\leftsquigarrow} D2$ which stands for ``$\rm A1^{-1}= (ps$-I$)^{-1} \circ (\rm D2)$", following Figure \ref{D2toA2} inversely. 
We will use similar notations also for other types of the MP moves.

\begin{lem}\label{Lem1}
We have the following relations among the MP moves and their inverses.
\begin{figure}[H]
    \centering
    \includegraphics[scale=0.8]{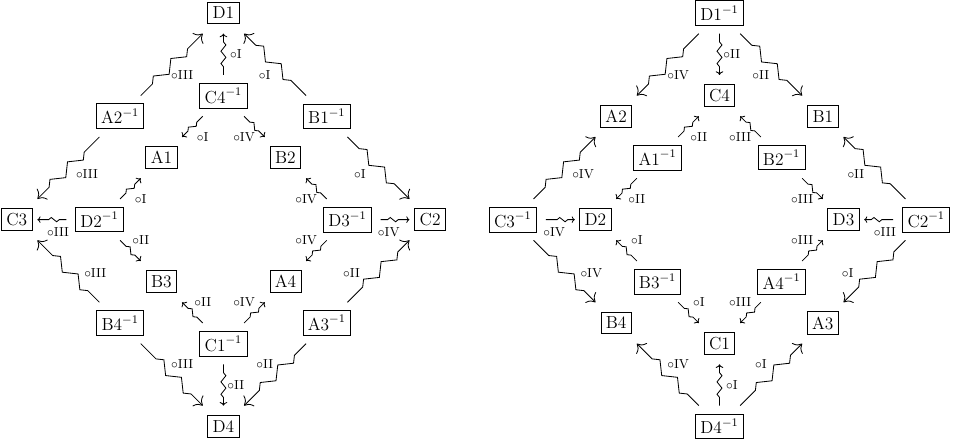}
    \caption{Relations among MP moves and their inverses.}
    \label{Relations among MP moves}
\end{figure}
\end{lem}

\begin{proof}
The proof for $\rm A \leftsquigarrow  D^{-1}$ type, i.e., the proof for $\rm  A2 \leftsquigarrow D1^{-1}$, $\rm  A3 \leftsquigarrow   D4^{-1}$, $\rm A4 \leftsquigarrow  D3^{-1}$, are similar to that for 
$\rm A1 \leftsquigarrow D2^{-1}$ in figure \ref{D2toA2} after adjusting the orientations of the two edges, see Figure \ref{MPPS1}.

\begin{figure}[H]
    \centering
    \includegraphics[scale=0.9]{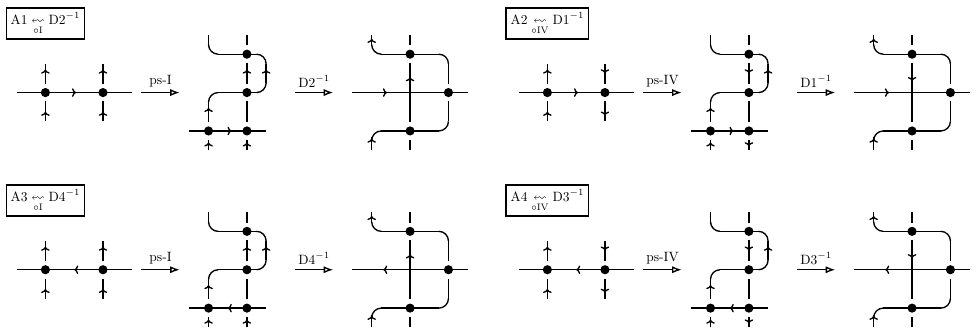}
    \caption{Proof of $\rm  A\protect\leftsquigarrow D^{-1}$ type.}
    \label{MPPS1}
\end{figure}

\begin{figure}[H]
    \centering
    \includegraphics[scale=0.9]{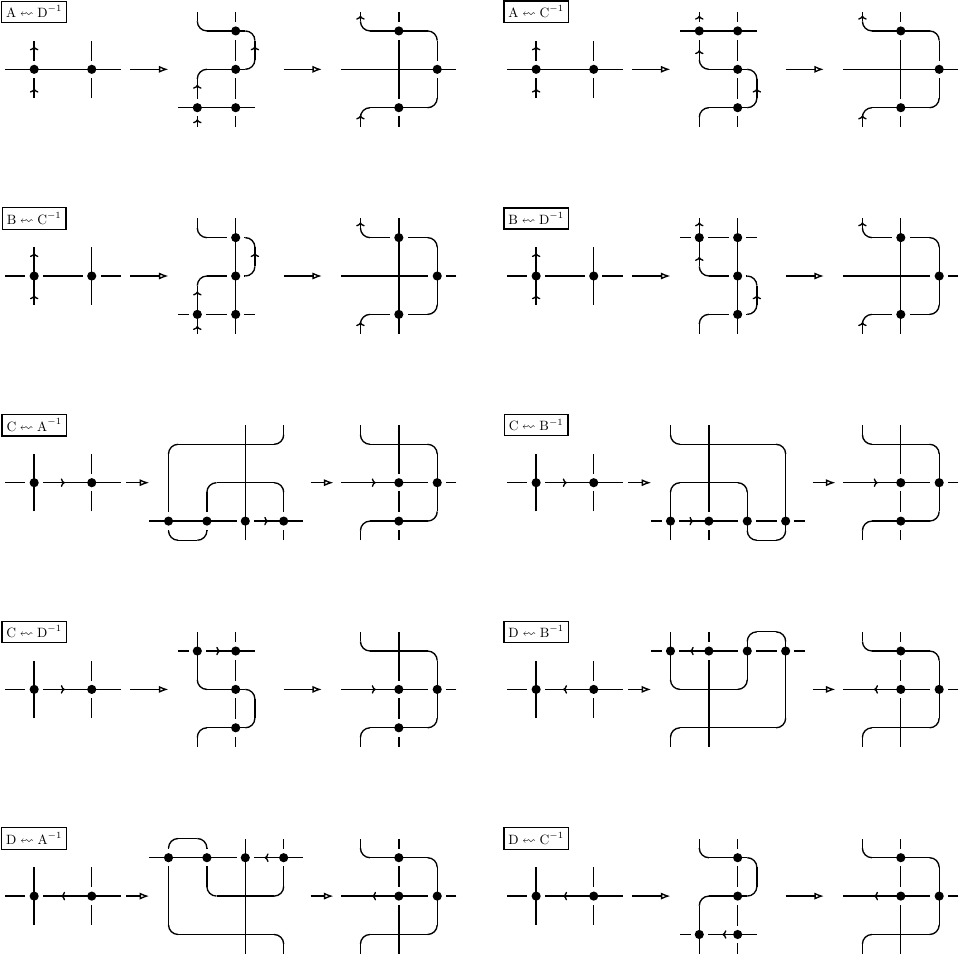}
    \caption{Proof of the diagrams in  Figure \ref{Relations among MP moves}.}
    \label{Proof of Main theorem}
\end{figure}

Similarly, the proof of the types $\rm  A\leftsquigarrow C^{-1}$, $\rm  B\leftsquigarrow C^{-1}$, $\rm B \leftsquigarrow D^{-1}$, $\rm C \leftsquigarrow A^{-1}$, $\rm C \leftsquigarrow B^{-1}$, $\rm  C\leftsquigarrow D^{-1}$, $\rm D \leftsquigarrow A^{-1}$, 
 $\rm D \leftsquigarrow B^{-1}$, 
 $\rm D \leftsquigarrow C^{-1}$  
  are given as in Figure \ref{Proof of Main theorem}, where the types of the MP moves and  the pure sliding moves follow the arrows in Figure \ref{Relations among MP moves}. Thus, we have the assertion.

\end{proof}

\begin{rem}
In the diagram in Figure \ref{Relations among MP moves}, there are luck of symmetries such that there are no arrow between moves of A-type and B-type. This is because of the asymmetry of the pure sliding moves and MP moves. If we consider including bumping moves (cf. \cite[Figure 3.22, Figure 4.16]{BP}), the diagram becomes more symmetric.
\end{rem}

\begin{rem}
We cannot derive any MP move or MP move inverse in the left diagram in Figure \ref{Relations among MP moves} from ones in the right diagram, and vice versa. This is because the number of positive true vertices minus the number of negative ones does not change under the pure sliding moves, and increases (resp. decreases) by one under the MP moves and the MP move inverses in the left (resp. right) diagram (cf. \cite[Remark 3.5.4]{BP}).
\end{rem}

\begin{proof}[Proof of Theorem \ref{th1}]
In Figure \ref{Relations among MP moves}, observe that each pair of the MP moves at the same location in the left and right diagrams are inverses to each other, and the arrows at the same location in the left and right diagrams point in opposite directions. Recall that $\rm A1 \underset{\circ \mathrm{I}}{\leftsquigarrow} D2^{-1}$ implies $ \mathrm{A1}^{-1}\underset{\mathrm{I}^{-1} \circ}{\leftsquigarrow} \mathrm{D2}$, which is also true for other relations. Thus, we can transport each arrow of type $\underset{\circ \mathrm{X}}{\leftsquigarrow}$ being changed to the arrows of type $\underset{\mathrm{X}^{-1} \circ}{\leftsquigarrow}$ from the left diagram to the corresponding place in the right diagram, and vice versa. Consequently, all MP moves or their inverses in the same diagram are equivalent up to the pure sliding moves and their inverses, as shown in Figure \ref{Relations among MP moves2}. Here, $a \leftrightsquigarrow b$ stands for $a \underset{\circ \mathrm{X}}{\leftsquigarrow} b$ and $a \underset{\mathrm{Y}^{-1} \circ}{\leadsto} b$, or $a \underset{\mathrm{Y}^{-1} \circ}{\leftsquigarrow} b$ and $a \underset{\circ \mathrm{X}}{\leadsto} b$, where each X and Y stands for one of ps-I -- ps-IV. \begin{figure}[H]
    \centering
    \includegraphics[scale=0.8]{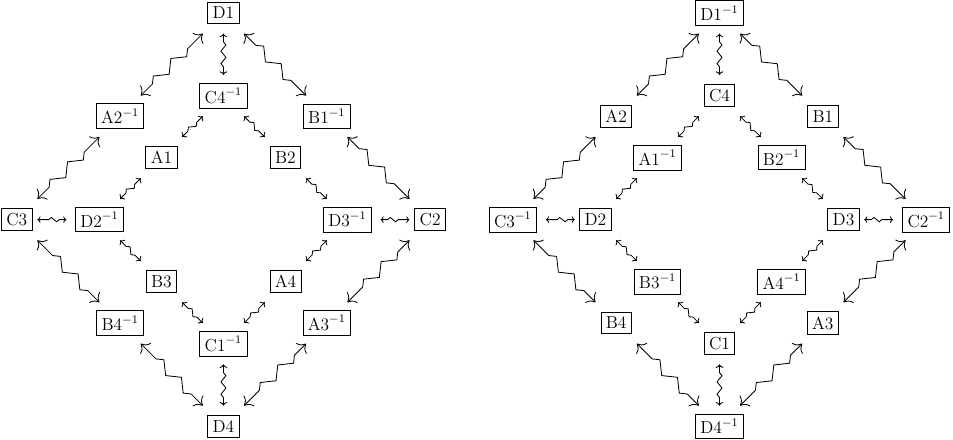}
    \caption{Relations among MP moves and their inverses up to the pure sliding moves and their inverses.}
    \label{Relations among MP moves2}
\end{figure}
 
Observe that each type of the MP moves is either in the left or right diagram, and its inverse is in the other diagram. In particular, from $\rm D2$ and $\rm D2^{-1}$,  we can obtain all MP moves and their inverses. This completes the proof of Theorem \ref{th1}.\end{proof}

\subsection{Local pure sliding moves and refinement of Theorem \ref{th1}} \label{localPS}
By taking the above argument forward, we can refine Theorem \ref{th1}  as follows. 
See the following pure sliding move which has already appeared in Figure \ref{D2toA2}.
\begin{figure}[H]
    \centering
    \includegraphics[scale=0.7]{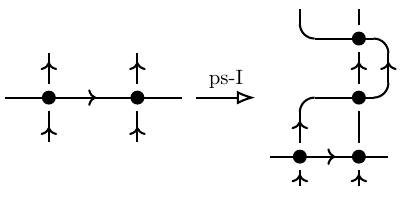}
    \caption{Local pure sliding move appearing in $\rm A1 \protect\leftsquigarrow D2^{-1}$. }
    \label{D2toA22}
\end{figure}
Recall that the two true vertices in the left picture enable us to check the condition (PS-I) locally as in Figure \ref{A1circ}. 
Similarly, there are 40 pure sliding moves in the proof of Lemma \ref{Lem1} appearing as the first move in each sequence in Figure \ref{Proof of Main theorem}. We call these pure sliding moves \textit{local pure sliding moves}.

\begin{prop} \label{localth1}
There is a set consisting of 18 local pure sliding moves such that each MP move is derived as a sequence of any MP move and 15 pure sliding moves in the set, and their inverses.
\end{prop}

\begin{proof}
We initially select 18 arrows from the diagrams in Figure \ref{Relations among MP moves}, as depicted in Figure \ref{Relations among MP moves3}. 
We can then convert each arrow of type $\underset{\circ \mathrm{X}}{\leftsquigarrow}$ into arrows of type $\underset{\mathrm{X}^{-1} \circ}{\leftsquigarrow}$ at corresponding locations in the other diagram, using a similar argument as in the proof of Theorem \ref{th1}. Subsequently, we obtain oriented paths consisting of arrows, allowing us to travel from any MP move to another. Let's consider an arbitrary MP move, say D2 for example. We can remove $3$ additional arrows, highlighted in gray in Figure \ref{Relations among MP moves3}, while maintaining connected paths from D2  or $\rm D2^{-1}$. This process of removing $3$ additional arrows can be applied to each MP move.  Thus, we have the assertion.

\begin{figure}[H]
    \centering
    \includegraphics[scale=0.8]{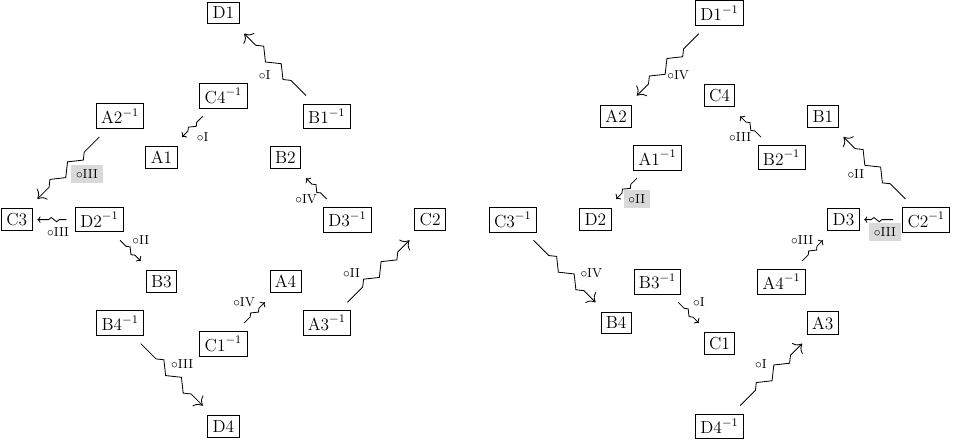}
    \caption{Paths by local pure sliding moves.}
    \label{Relations among MP moves3}
\end{figure}

\end{proof}
\begin{cor}
Each MP move is derived as a sequence of the primary MP move and 15 local pure sliding moves, and their inverses.
\end{cor}
\begin{proof}
Recall from Figure \ref{D2 and primary MP} that the primary MP move is equivalent to the MP move D2, which, together with Proposition \ref{localth1}, confirms the assertion. 
\end{proof}

\section{Integral normal o-graphs and framed $3$-manifolds}\label{Section;frame}
 In this section we follow the notation in \cite{MST2}.

\subsection{Branched spines and framed $3$-manifolds}\label{subsec: framing} 
In what follows we sometimes use metrics of $3$-manifolds for convenience, while the results do not depend on them. Let $M$ be a closed $3$-manifold. 
A \textit{{framing}} $(v_1, v_2, v_3)$ of $M$ is a trivialization of the tangent bundle $TM$ such that $v_i\perp v_j$ for $i,j=1,2,3$ and $i\neq j$.
We assume that the orientation of $M$ induced from the framing matches with the existing one.
Two framed $3$-manifolds $(M, v_1, v_2, v_3)$ and  $(M',v_1^{\prime},v_2^{\prime},v_3^{\prime})$ are \textit{{equivalent}} if there exists a diffeomorphism $h\co M\to M^{\prime}$ such that $(h_*v_1, h_*v_2, h_*v_3)$ is homotopic through framings to $(v_1^{\prime},v_2^{\prime},v_3^{\prime})$. We denote by $\mathcal{M}_{\text{fram}}$  the
 set of equivalence classes of oriented framed $3$-manifolds.
Notice that we can obtain the third vector $v_3$ of a framing from $v_1$, $v_2$ and the orientation of $M$.
Thus, we represent a framed $3$-manifold as $(M, v_1, v_2)$ specifying only the first two vectors $v_1, v_2$.

Let $P$ a closed branched spine representing a  combed $3$-manifold $(\widehat M(P), \widehat v_1(P))$. Observe that $v_1$ is perpendicular to $P$, thus $(\mathbb{R}\widehat v_1)^{\perp}$ is the ``tangent bundle" of $P$, where tangency makes sense because $P$ is branched. The Euler class $\mathcal{E}\in H^2(\widehat M;\mathbb{Z})$ of $(\mathbb{R}\widehat v_1)^{\perp}$ is the obstruction to the existence of a framing of $M$ which extends $v_1$, and is represented by the \textit{Euler cochain} $c_P\in C^2(P;\mathbb{Z})$ defined in \cite{BP}*{Proposition 7.1}.
We will explain the definition of the Euler cochain and its nature in Section \ref{Eulercochain} before we study framing explicitly.

Given a 1-cochain $x\in C^1(P;\mathbb{Z})$ satisfying $\delta x= c_P$, we can define the second vector $\widehat v_2(P,x)$ on $(\widehat M(P), \widehat v_1(P))$ as follows.
We first define the second vector $v_2(P,x)$ on $M(P)$ near the 0-cells as in Figure \ref{sencond_combing_around_vertices}. 
\begin{figure}[H]
    \centering
    \includegraphics[scale=1]{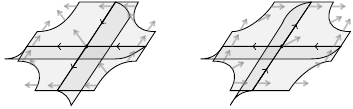}
    \caption{Second vector $v_2$ of framing near vertex of type $+$ (left) and of type $-$ (right).}
    \label{sencond_combing_around_vertices}
\end{figure}
\noindent Then we extend it over $S(P)$ so that the 1-cochain $x$ defines the rotation number of $v_2(P,x)$ relative to the tangent vector of $e$ having the canonical orientation coming from the branching of $P$.
We extend it over $2$-cells, where we can do this because of the boundary condition $\delta x= c_P$ (\cite{BP}*{Proposition 7.2.1}), and since $\pi_2(S^1)=0$, the extension over the 2-cells are unique up to homotopy.
Since $\pi_2(S^1)=\pi_3(S^1)=0$, the second vector $v_2(P,x)$ defined on $M(P)$ extends uniquely to $\widehat{M}(P)$.  Thus, we get the unique framed $3$-manifold $(\widehat{M}(P), \widehat v_1(P), \widehat v_2(P,x))$ up to equivalence. 
Conversely, every framed $3$-manifold is obtained in this way \cite{BP}*{Proposition 7.2.4}.

\subsection{Integral normal o-graphs}\label{sec:BP-diagrams}

An \textit{{integral normal o-graph}} is a normal o-graph  with an integer  weight attached on each edge.
Note that the correspondence of normal o-graphs and branched polyhedrons implies the correspondence of integral normal o-graphs and pairs $(P,x)$ of  a branched polyhedron $P$  and  its $1$-cochain $x\in C^1(P;\mathbb{Z})$. 
A \textit{{framed integral normal o-graph}} is an integral closed normal o-graph satisfying $\delta x= c_P$,  which represents a framed $3$-manifold. 
For example, see Figure \ref{EX1} for framed integral normal o-graphs representing $S^3$ and the lens space $L(2,1)$ with certain framings \footnote{For $S^3$ it extends the combing induced by the Hopf fibration, and for the lens space $L(2,1)$ it extends the canonical combing induced by its Seifert fibered structure \cite{EndoIshii}.}.

\begin{figure}[ht]
    \centering
    \includegraphics{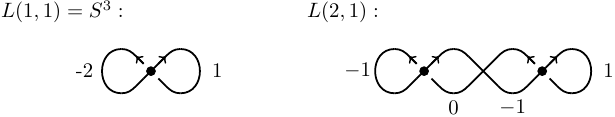}
    \caption{Framed integral normal o-graphs representing $S^3$ and $L(2,1)$ with certain framings.}\label{EX1}
\end{figure}

We consider three types of moves on framed integral normal o-graphs, the \textit{integral $0$-$2$ move} in Figure \ref{fig:0-2 move}, the  \textit{integral MP-move} in Figure \ref{fig:MP-move}, and the \textit{integral H-move} in Figure \ref{fig:H-move}. Here, in Figure \ref{fig:MP-move} the orientations of the non-oriented edges are arbitrary if they match before and after the move, and   in Figure \ref{fig:MP-move} and \ref{fig:H-move} if there are multiple weights on an edge after the move, the weights should be added in the additive group $\mathbb{Z}$. These moves do not change the framing of associated $3$-manifolds \cite{BP, MST2}.

  \begin{figure}[H]
    \centering
    \includegraphics[scale=0.7]{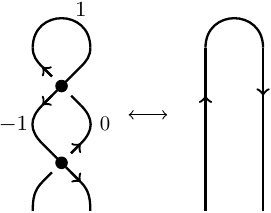}
    \caption{Integral $0$-$2$ move.}
    \label{fig:0-2 move}
\end{figure}
\begin{figure}[H]
    \centering
    \includegraphics[scale=0.8]{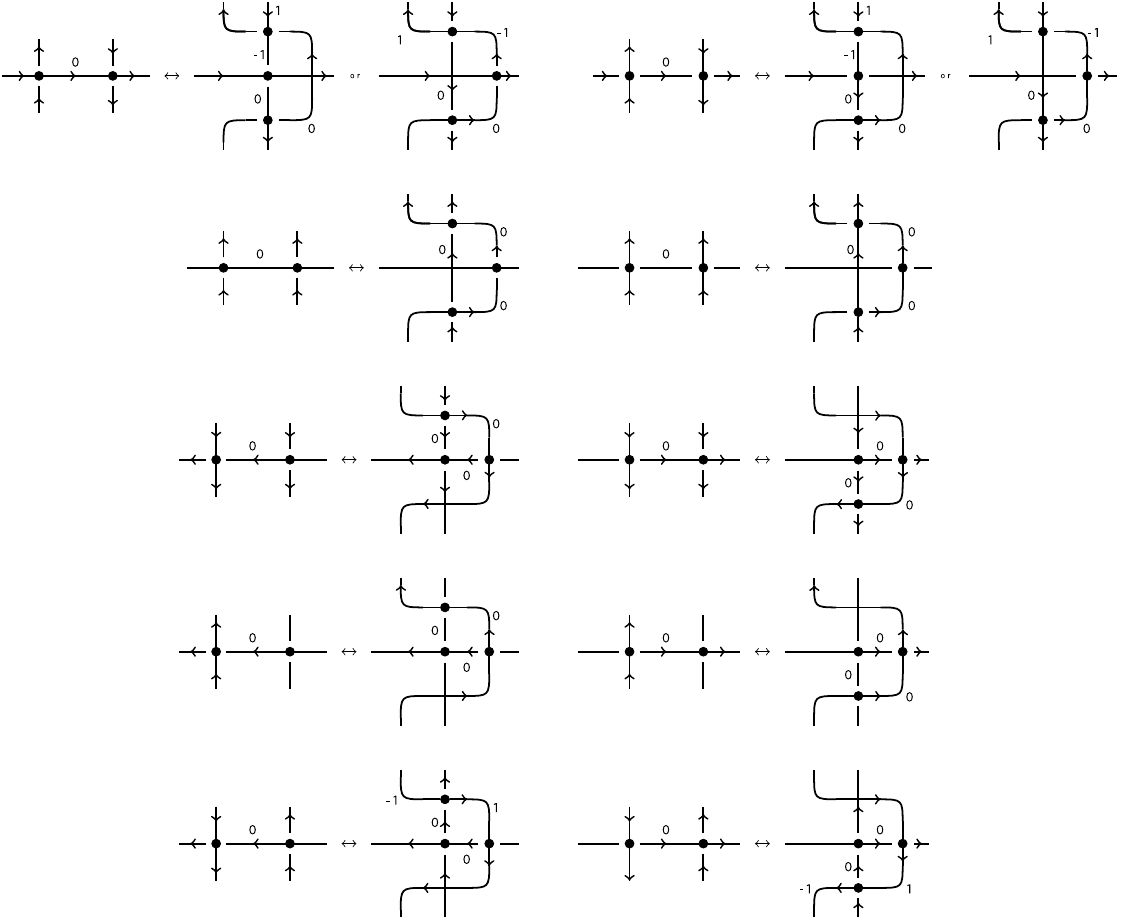}
    \caption{Integral MP-move. Orientation of each non-oriented edge is arbitrary if it matches before and after  move.}
    \label{fig:MP-move}
\end{figure}

\begin{figure}[H]
    \centering
    \includegraphics{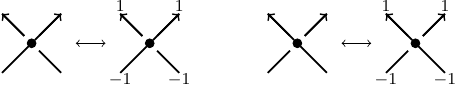}
    \caption{Integral H-move.}
    \label{fig:H-move}
\end{figure}

\subsection{Framed normal o-graphs and framed $3$-manifolds}\label{Section;framed_o-graphs}

The argument in the previous section shows that for each closed normal o-graph $\Gamma$ and its $1$-chain $x\in C^1(\Gamma; \mathbb{Z})$ such that $\delta x=c_{P}$ we have a framed $3$-manifold $\Phi_{\rm fram}(\Gamma,x)$, and each equivalent class of  framed $3$-manifolds is obtained  in this way.
In fact, $\Phi_{\rm fram}(\Gamma, x)$ depends only on $(\Gamma,  [x]_2)$, where $[x]_2\in C^1(\Gamma, \mathbb{Z}_2)$ is the projection of $x\in C^1(\Gamma, \mathbb{Z})$   (\cite[Theorem ]{BP}). 

A \textit{{framed normal o-graph}} \cite{BP} is the image $(\Gamma, [x]_2)$ of a framed integral normal o-graph $(\Gamma, x)$ by the projection on the second factor.
We denote by $\mathcal{G}_{\rm fram}$ the set of framed normal o-graphs, and define the  \textit{framed $0$-$2$ move},  the \textit{framed MP-move}, and the \textit{framed H-move} on $\mathcal{G}_{\rm fram}$ as the projections of the corresponding integral moves on integral normal o-graphs. 
With these conventions, we have a surjective map
\begin{align*}
\Phi_{\rm fram} \co \mathcal{G}_{\rm fram} \to \mathcal{M}_{\rm fram}.
\end{align*}

\begin{prop}[{Benedetti-Petronio \cite[Theorem 1.4.3]{BP} }]\label{BPfram}
The equivalence relation defined by  $\Phi_{\rm fram}$
is generated by the framed $0$-$2$ move, the framed MP moves, and the framed H-move.
\end{prop}

\subsection{Euler cochain $c_P$}\label{Eulercochain}
In the next section we will define pure sliding moves on framed normal o-graphs,
where the well-definedness  (Proposition \ref{IPS move and framing}) requires
arguments including the nature of the Euler cochain $c_P$. 
For this reason we explain $c_P$ in this section, see \cite{BP} for more details.

Let $P$ be a branched polyhedron. The Euler cochain of $P$ is given by  
\begin{align*}
    c_P = -\sum_i (1-n_i/2)\hat{\Delta}_i \in C^2(P;\mathbb{Z}),
\end{align*}
where $\hat{\Delta}_i\in C^2(P;\mathbb{Z})$ is the dual of a connected component ${\Delta}_i$ (2-cells) of $D(P)$, and the dots $n_i$ is the total number of solid dots as in Figure \ref{fig:c_p def} on the boundary of ${\Delta}_i$.

\begin{figure}[H]
    \centering
     \includegraphics[scale=0.8]{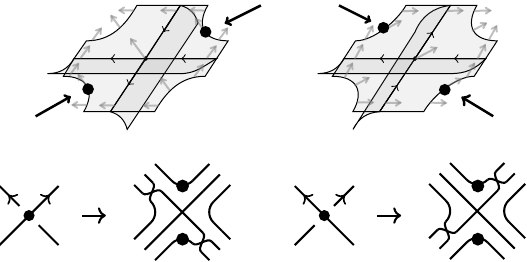}
        \caption{Solid dots on boundary of $\Delta_i$.}
    \label{fig:c_p def}
\end{figure}
 Let $v$ be the vector field on $S(P)$ tangent to $P$, defined around $V(P)$ by the gray arrows in Figure \ref{sencond_combing_around_vertices} (and in the top of Figure \ref{fig:c_p def}), while $v$ outside $V(P)$ is defined to not be tangent to $S(P)$.
The Euler cochain $c_P$ is constructed as the obstruction to extending $v$ to $P$ as follows: In Figure \ref{fig:c_p def}, the solid dots on ${\Delta}_i$ indicate points where $v$ is tangent to $\partial {\Delta}_i$. When counting the rotation number of $v$ on $\partial {\Delta}_i$ relative to the tangent vector of $\partial {\Delta}_i$, each time we pass through a dot, the vector $v$ rotates by $-1/2$. The total rotation number is thus $-n_i/2$. Consequently, the vector field $v$ extends to $2$-cells if and only if $-n_i/2=-1$ for each $i$, which is equivalent to $c_P=0$.
In general, for $x\in C^1(P; \mathbb{Z})$, let $v_x$ be the vector field on $S(P)$ obtained by rotating $v$ using $x$ as described in Section \ref{subsec: framing}. Then, $\delta x -c_P$ represents the obstruction to extending $v_x$ to $2$-cells.

\subsection{Integral pure sliding moves on integral normal o-graphs}\label{Section;IPS}
We introduce the \textit{integral pure sliding moves} I -- IV as in Figure \ref{the integral  pure sliding move}, where we can perform them only when the underlying normal o-graphs satisfy the conditions (PS-I)--(PS-IV), respectively.

\begin{figure}[H]
    \centering
    \includegraphics[scale=0.85]{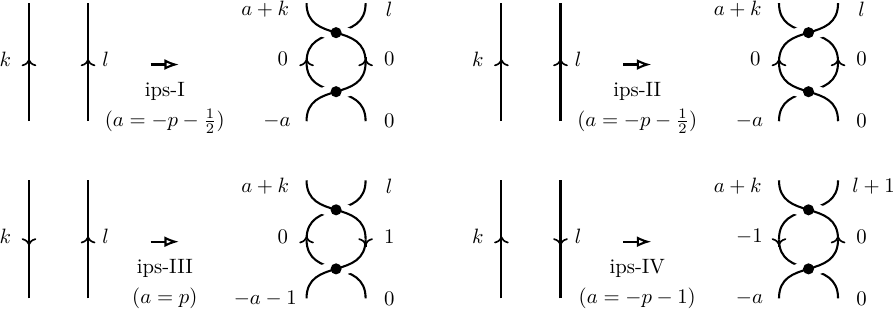}
    \caption{Integral pure sliding move ips-I -- ips-IV.}
    \label{the integral  pure sliding move}
\end{figure}

\noindent Here, $k$ and $l$ represent integer weights assigned to parts of edges. If there are multiple weights on an edge, the integer weight of the edge is the sum of these weights (for instance, the case when two open edges above are parts of one edge). Additionally, $p \in \mathbb{Z}/2$ is an integer or half-integer defined as follows.

Let $P$ be a branched spine of $M$, and $x \in C^1(P; \mathbb{Z})$ be a cochain representing a framing of $M$. Consider the $2$-cell $\Delta$ on which we perform the pure sliding move. The boundary $\partial \Delta$ is oriented as usual, traveling along edges of $P$ (potentially visiting one edge multiple times). We can then determine the rotation number of the second vector $v_2$ of framing on $\partial \Delta$ using $x \in C^1(P; \mathbb{Z})$, as depicted in the left image of Figure \ref{IPS move condi}.

\begin{figure}[H]
    \centering
    \includegraphics[scale=0.8]{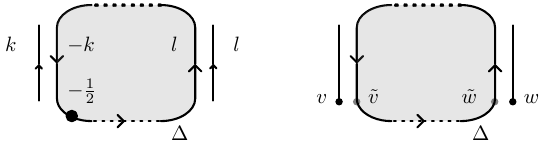}

    \caption{How we count rotation number on $\Delta$ (left) and points $\tilde v$ and $\tilde w$ on $\partial \Delta$ (right).}
    \label{IPS move condi}
\end{figure}

\noindent 
Namely, when $\partial \Delta$ passes along an edge $e$ in the direction induced by $\Delta$, if $e$ is oriented in the same direction as $\partial \Delta$ proceeds, then the rotation number on that part is $x(e)$; otherwise, it is $-x(e)$.
When $\partial \Delta$ passes along a vertex, if there is a solid dot as shown in Figure \ref{fig:c_p def}, then we count $-1/2$.

Refer to the defining picture of ips moves in Figure \ref{the integral pure sliding move}. Let $e_1$ and $e_2$ be the left and right open edges, respectively, on which we perform an ips move. We can assume that on the edges where $e_1$ and $e_2$ reside, the second vector $v_2$ of the framing rotates $k$ and $l$ times inside $e_1$ and $e_2$, respectively, while outside it goes transversely to the edge from left to right following the edge's orientation (the opposite direction of the branching). Then, $p$ represents the rotation number of $v_2$ on $\partial \Delta$ from the bottom  of $e_1$ (positioned as shown in Figure \ref{the integral pure sliding move}) to the bottom of $e_2$. 

We can calculate $p$ as follows.  If the bottoms of $e_1$ and $e_2$ are connected without a vertex, then  $p=0$. If there are vertices, let $v$ and $w$ be the vertices (potentially the same vertex) that we encounter when we reach the bottom of $e_1$ and $e_2$, respectively, see the right image of Figure \ref{IPS move condi}. Let $\tilde{v}$ and $\tilde{w}$ be the points (again, they could be the same point) on $\partial \Delta$ attached to $v$ and $w$, respectively, at the bottom of the sliding region (the pre-image is not globally unique). Then, $p$ is the rotation number of $v_2$ on $\partial \Delta$ from $\tilde{v}$ to $\tilde{w}$, where if there is a solid dot at $\tilde{v}$ or $\tilde{w}$, we count $-1/2$ for each of them.

\begin{prop}\label{IPS move and framing}
Let $\Gamma$ and $\Gamma'$ be two framed  integral normal o-graphs which are related by the integral pure sliding moves.
Then the framed $3$-manifolds induced by $\Gamma$ and $\Gamma'$ are equivalent.
\end{prop}
\begin{proof}We prove the assertion for ips-I. For the other moves the proof is similar.

 Let $(\Gamma,x)$ be a framed integral normal o-graph and $(\Gamma',x')$ be the integral normal o-graph obtained from $(\Gamma,x)$ by ips-I. Let $\Delta$ be the $2$-cell on which we perform ips-I.
We first divide $\Delta$ into five disks $\Delta_1$--$\Delta_5$ as the left picture  in Figure \ref{Decompose the disk}. Then we arrange the second vector $v_2$ of the framing on $\Delta$ by homotopy which is the identity on the boundary, as in the right picture, where the integer of each edge represents the rotation number relative to the orientation of the boundary of the disks, and each white dot represents the tangency points which has the rotation number $-1/2$. We can do this because on each disk the total rotation number on the boundary is $-1$: for $\Delta_2$, $\Delta_3$, and $\Delta_4$, this follows by straightforward calculation; for $\Delta_5$, it follows from the definition of $a=-p-1/2$; and for $\Delta_1$,  it follows from the definition of $a$ and the equation derived from the initial condition that the rotation number on $\delta \Delta$ is $-1$.

\begin{figure}[H]
    \centering
    \includegraphics[scale=0.8]{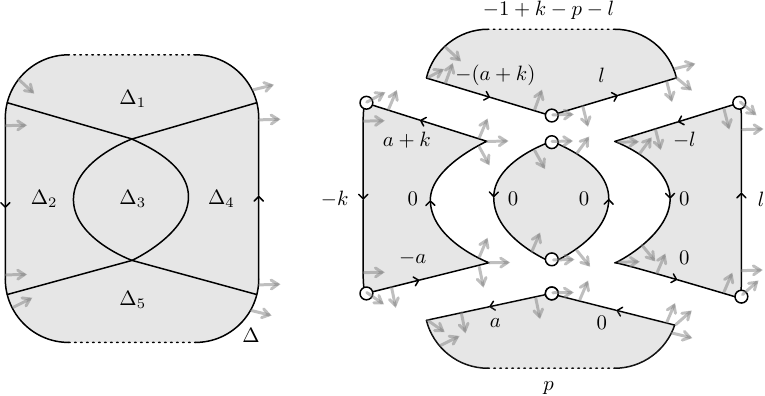}
    \caption{Decomposition of $\Delta$ (left) and second vector $v_2$ of  framing on $\Delta$ (right).}
    \label{Decompose the disk}
\end{figure}

Let $M$ be the framed $3$-manifold obtained by $(\Gamma,x)$. Near the sliding region, we consider a planar figure, and we assume that the branched spine is almost horizontal so that the second vector $v_2$ of the framing is obtained by the pullback of the projection to the horizontal plane. Under this setting, we can pull back $v_2$ as depicted in the right picture in Figure \ref{Decompose the disk}, thus defining the framing of $M$ near the sliding region. Note that this framing is a representative for both the homotopy classes of framings obtained from $(\Gamma,x)$ and $(\Gamma',x')$. Thus, we have the assertion.
\end{proof}
\begin{rem}\label{nonlocal}
In \cite[Section 7.3]{BP}, the authors demonstrated results similar to Proposition \ref{IPS move and framing}, showing that the framed $0$-$2$ move and the framed MP move do not alter the framing of closed $3$-manifolds. They illustrated local framing changes near the sliding regions explicitly by projecting them onto a plane \cite[Figure 7.7--7.9]{BP}. This is feasible because they can assume the cochain $x \in C^1(P, \mathbb{Z})$ affects regions far from the sliding regions, allowing to take an explicit framing near sliding regions.
However, in our case, this approach is not applicable in Figure \ref{Decompose the disk} because the framing near the sliding region depends on the framing on the boundary of $\Delta_1$ and $\Delta_5$ out of the sliding region.
\end{rem}

\subsection{Main results for framed $3$-manifolds}\label{Section;Main_results_framing}
We define the primary integral MP move as follows.
\begin{figure}[H]
    \centering
    \includegraphics[scale=0.7]{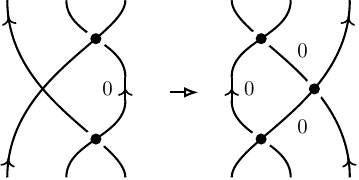}
    \caption{Primary integral MP move.}
\end{figure}

\begin{thm}\label{thZ}
Each of the 16 integral MP move in Figure \ref{fig:MP-move} is derived as a sequence of the primary integral MP move, the integral pure sliding moves,  the integral H-move and their inverses.
\end{thm}

\begin{proof}[ Proof of Theorem \ref{thZ}]
The proof follows a similar path to the proof of Theorem \ref{th1}, once integer weights are attached. The detailed proof is left to the readers.
\end{proof} 

By taking modulo $2$ on each integer weight, we have the following.
\begin{cor}\label{thZ2}
Each of the 16 framed MP move is derived as a sequence of the primary framed MP move, the  framed pure sliding  moves, the framed H-move and their inverses.
\end{cor}

Note that the integral $0$-$2$ move is an integral pure sliding move.
By Proposition \ref{BPfram}, Proposition \ref{IPS move and framing} and Theorem \ref{thZ2}, we have the following representation of framed $3$-manifolds.
\begin{cor}\label{cofram}
The equivalence relation defined by the surjective map
$\Phi_{\rm fram} \co \mathcal{G}_{\rm fram} \to \mathcal{M}_{\rm fram}$
is generated by the primary framed MP move, the framed pure sliding moves and the framed H-move.
\end{cor}

\subsection{Local integral pure sliding moves and refinement of Theorem \ref{thZ}} \label{integrallocalPS}
In this section we give a refinement of Theorem \ref{thZ}, which generalizes Proposition \ref{localth1}.
The integral pure sliding moves are local in the sense that the normal o-graph does not change outside the sliding region. However, they are not local in the sense that we need global information about the $2$-cells to perform the moves (as with pure sliding moves on normal o-graphs), and moreover, the resulting cochain also depends on the global information of the original cochain, as noted in Remark \ref{nonlocal}.

We define \textit{local integral pure sliding moves} by attaching integer weights to the 40 local pure sliding moves defined in Section \ref{localPS}. 
\begin{cor}
Each integral MP move is derived as a sequence of the primary integral MP move and 15 local integral pure sliding moves and their inverses.
\end{cor}
\begin{proof}
The proof is similar to that of Proposition \ref{localth1}. 
\end{proof}

\section{Spin normal o-graphs and spin $3$-manifolds}\label{Section;spin}
A spin structure of a $3$-manifold with a standard spine $P$ can be thought of as a homotopy class of framing over the 1-skeleton that extends over the $2$-skeleton.
Let $\Gamma$ be a closed normal o-graph. 
Spin structures on $M(\Gamma)$ correspond bijectively  to cochains $z\in C^1(P(\Gamma), \mathbb{Z}_2)$ such that $\partial z = (c_p)_2\in C^2(P(\Gamma), \mathbb{Z}_2)$, where two cochains are viewed as equivalent if they differ by a $1$-coboundary \cite[Proposition 7.4.2]{BP}. 
A  \textit{{spin normal o-graphs}} is a pair $(\Gamma, y)$ of a normal o-graph $\Gamma$ and  $y \in C^1(\Gamma, \mathbb{Z}_2)$ such that $\partial y =(c_p)_2$. We denote by $\mathcal{G}_{\rm spin}$ the set of spin normal o-graphs.
Thus, we have a surjective map
\begin{align*}
\Phi_{\rm spin} \co \mathcal{G}_{\rm spin} \to \mathcal{M}_{\rm spin}.
\end{align*}
We define the \textit{framed $0$-$2$ move}, the \textit{framed MP moves}, the \textit{framed H-move} on $\mathcal{G}_{\rm spin}$ in a similar way to these on $\mathcal{G}_{\rm fram}$.

\begin{prop}[{Benedetti-Petronio \cite[Theorem 1.4.4]{BP} }]\label{BPspin}
The equivalence relation defined by  $\Phi_{\rm spin}$ is generated by the framed $0$-$2$ move, the framed MP moves, the framed H-moves, and the \textit{framed CP move} defined in \cite[Figure 1.9]{BP}.
\end{prop}

We define framed pure sliding moves on $\mathcal{G}_{\rm spin}$ in a similar way to these on $\mathcal{G}_{\rm fram}$.
\begin{prop}\label{IPS move and spin}
Let $(\Gamma ,y)$ and $(\Gamma', y')$ be two spin normal o-graphs which are related by the framed pure sliding moves.
Then the induced spin $3$-manifolds are equivalent.
\end{prop}
\begin{proof}
Using the similar assumption of the proof of Proposition \ref{IPS move and framing}, we can represent each spin structure near the sliding region by a framing on $\Delta$ whose first vector on the boundary of the disks $\Delta_1$--$\Delta_5$ is $v_1=v(P)$ (cf. \cite[Lemma 7.4.1]{BP}) and the second vector on the boundary of the disks $\Delta_1$--$\Delta_5$ is given by the rotation number on each edge as in Figure \ref{Decompose the disk}, where the integer weights are considered in $\mathbb{Z}_2$.
This framing is a representative for the spin structures obtained from  $(\Gamma,y)$ and $(\Gamma',y')$. Thus, we have the assertion.

\end{proof}

\begin{cor}\label{cospin}
The equivalence relation defined by the surjective map
$\Phi_{\rm spin} \co \mathcal{G} \to \mathcal{M}_{\rm spin}$
is generated by the primary framed MP move, the framed pure sliding move, the framed H-move and  the framed CP move.
\end{cor}
\begin{proof}
The assertion follows immediately from  Corollary \ref{thZ2}, Propositions \ref{BPspin} and  \ref{IPS move and spin}.
\end{proof}

\begin{appendices}

\section{Symmetry of moves}\label{Section;symmetry}
In this section, we investigate the symmetries of the MP moves. In what follows, we assume, for simplicity, that the normal o-graph $\Gamma$ is closed, although the argument applies to normal o-graphs in general. 
Let $P(\Gamma)$ be the branched polyhedron represented by $\Gamma$, as explained in Section \ref{Normal o-graph}.
Recall from Section \ref{Closed normal o-graph} the equivalent class of a combed  $3$-manifold $(\widehat M(\Gamma), \widehat v(\Gamma))$ obtained from the branched spine $P(\Gamma)$. For brevity, we will denote it by $(M(\Gamma), v(\Gamma))$, omitting the hats. We may occasionally use the same notation for a representative of the equivalent class.

Let $\Gamma^{*}$, $\overrightarrow{\Gamma}$, and $\Gamma^{\times}$ be the normal o-graphs obtained from a normal o-graph $\Gamma$ by reflecting its diagram, reversing its orientation, and changing the sign of every crossing, respectively. 
Note that each of these operations is involutive. 

For a combed $3$-manifold $(M(\Gamma),v(\Gamma))$, let $(M^{\text{op}}(\Gamma),v^{\text{op}}(\Gamma))$ be the equivalent class of combed $3$-manifolds defined by the image of $(M(\Gamma),v(\Gamma))$ under an orientation-reversing diffeomorphism.
For a combing $v$ of a $3$-manifold we denote by $-v$ the combing obtained by multiplying $(-1)$ to ${v}$ at each point. 

\begin{prop}
We have $({M}(\overrightarrow {\Gamma^*}),{v}(\overrightarrow {\Gamma^*})) =(M^{\text{op}}(\Gamma),v^{\text{op}}(\Gamma))$ and 
$(M(\Gamma^{\times}), v(\Gamma^{\times}))=(M^{\text{op}}(\Gamma),-v^{\text{op}}(\Gamma))$.
\end{prop}
\begin{proof}
We first consider $\Gamma^{\times}$. Recall from Section \ref{Normal o-graph} the construction of the branched spine $P(\Gamma)$ from a normal o-graph $\Gamma$, and recall from Figure \ref{fig:BStoBP} the correspondence between a crossing of  $\Gamma$ and a neighborhood (butterfly) of the corresponding true vertex in $P(\Gamma)$. When we change the sign of a crossing, the corresponding butterfly is reflected in $\mathbb{R}^3$ by the $xy$-plane on which the four horizontal $2$-cells in the butterfly are placed, see Figure \ref{fig:SC}.
\begin{figure}[H]
    \centering
    
    \includegraphics[scale=1]{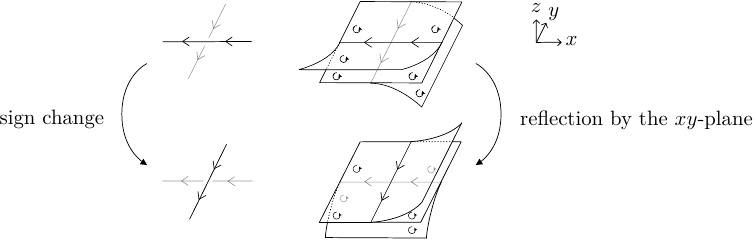}
    \caption{Sign change of normal o-graph and corresponding butterfly.}
    \label{fig:SC}
\end{figure}
\noindent This correspondence induces the orientation-reversing diffeomorphism $f\co (x,y,z) \mapsto (x,y,-z)$ between neighborhoods of  butterflies in ${M}(\Gamma)$ and ${M}(\Gamma^{\times})$, and we can extend it to whole $3$-manifolds.  Note that the push-foward $f_*(v(\Gamma))$ of the combing in ${M}(\Gamma)$ goes to $-z$ direction which is opposite to the combing of ${M}(\Gamma^{\times})$ that goes to $z$ direction.
Consequently we have $({M}(\Gamma^{\times}), v(\Gamma^{\times}))=(f(M(\Gamma)),-f_*(v(\Gamma))=(M^{\text{op}}(\Gamma),-v^{\text{op}}(\Gamma))$.

We then consider $(\overrightarrow {\Gamma^*})^{\times}$, performing all $3$ involutions. When we perform them around a crossing of a $\Gamma$, the corresponding butterfly in $P(\Gamma)$ is flipped, i.e.,  rotated $\pi$ degrees around the line in the $xy$-plane which we use when we reflect the diagram of the crossing. 
Here, the orientation of edges of the normal o-graph is reversed because the orientation of  $2$-cells is reversed after we flip it, according to the assumption for constructing the diagram that the orientation of the $2$-cell should be compatible with that of the $xy$-plane, see Figure \ref{fig:INV}.
\begin{figure}[H]
    \centering
    
    \includegraphics[scale=1]{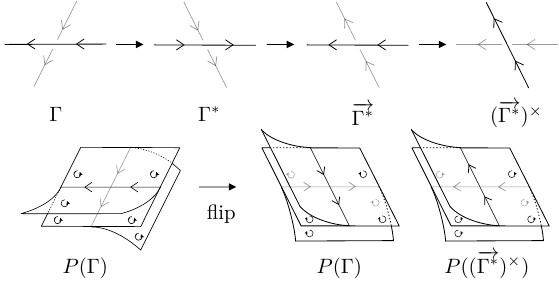}
    \caption{Performing all involutions on normal o-graph and resulting butterfly.}
    \label{fig:INV}
\end{figure}
\noindent This flip is extended to an orientation-preserving diffeomorphism of whole $3$-manifolds, and the combing is reversed because the orientation of the $2$-cell is reversed. Thus, we have $({M}((\overrightarrow {\Gamma^*})^{\times}),v((\overrightarrow {\Gamma^*})^{\times}))=(M(\Gamma),-v(\Gamma))$.

For $\overrightarrow {\Gamma^*}=\big((\overrightarrow {\Gamma^*})^{\times}\big)^{\times}$, the above two cases imply that $({M}(\overrightarrow {\Gamma^*}),{v}(\overrightarrow {\Gamma^*})) =(M^{\text{op}}(\Gamma),v^{\text{op}}(\Gamma))$, see Figure \ref{fig:Refrev}, which completes the proof. 
\begin{figure}[H]
    \centering
    
    \includegraphics[scale=1]{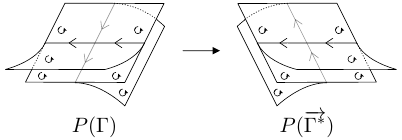}
    \caption{Butterflies in $P(\Gamma)$ and $P(\protect\overrightarrow{\Gamma^*})$.}
    \label{fig:Refrev}
\end{figure}
\end{proof}

We study the behavior of the 16 MP moves and pure sliding moves under the involutions  $\vec{\ }\circ {}^*$ (reflect and reverse the orientation)  and ${}^{\times}$ (change the sign of each crossing). Tables \ref{Symmetries of MP moves} and \ref{Symmetries of pure sliding moves} show the results of moves by the involutions. For example, if a normal o-graph $\Gamma_2$ is obtained from another normal o-graph $\Gamma_1$ by A1 on vertices $v_1, v_2$ of $\Gamma_1$, then $\overrightarrow {\Gamma_2^*}$ is obtained from $\overrightarrow {\Gamma_1^*}$ by A3 on the image of vertices $v_1, v_2$ by $\vec{\ }\circ {}^*$.

\begin{table}[H]
    \centering

\begin{tabular}{|c|c|c|c|c|c|c|c|c|c|c|c|c|c|c|c|c|}\hline
   & A1 & A2 & A3 & A4 & B1 & B2 & B3 & B4 & C1 & C2 & C3 & C4 & D1 & D2 & D3 & D4  \\ \hline
reflect ${}^*$ \& reverse $\vec{\ }$ & A3 & A4 & \textcolor[gray]{0.6}{A1} & \textcolor[gray]{0.6}{A2} & B3 & B4 & \textcolor[gray]{0.6}{B1} & \textcolor[gray]{0.6}{B2} & D1 & D2 & D3 & D4 &  \textcolor[gray]{0.6}{C1} & \textcolor[gray]{0.6}{C2} & \textcolor[gray]{0.6}{C3} &\textcolor[gray]{0.6}{C4}  \\ \hline
sign change ${}^{\times}$& B1 & B2 & B3 & B4 & \textcolor[gray]{0.6}{A1} & \textcolor[gray]{0.6}{A2} & \textcolor[gray]{0.6}{A3} & \textcolor[gray]{0.6}{A4} & D4 & D2 & D3 & D1 & \textcolor[gray]{0.6}{C4} & \textcolor[gray]{0.6}{C2} & \textcolor[gray]{0.6}{C3} & \textcolor[gray]{0.6}{C1} \\ \hline
\end{tabular}

    \caption{Symmetries of MP moves.}
    \label{Symmetries of MP moves}
\end{table}

\begin{table}[H]
    \centering

\begin{tabular}{|c|c|c|c|c|}\hline
 & ps-I & ps-II & ps-III & ps-IV \\ \hline
reflect ${}^*$ \& reverse $\vec{\ }$ & ps-I &ps-II &ps-III & ps-IV \\ \hline
sign change ${}^{\times}$ &ps-II & \textcolor[gray]{0.6}{ps-I} & ps-III& ps-IV\\ \hline
\end{tabular}

    \caption{Symmetries of pure sliding moves.}
    \label{Symmetries of pure sliding moves}
\end{table}

By using the above argument, we can simplify the proof of  Lemma \ref{Lem1}.
One can check that the right diagram in Figure \ref{Relations among MP moves}
is equal to the image of the left diagram by $\vec{\ }\circ {}^*$,
and also the image by ${}^{\times}$.

In  Proposition \ref{localth1},  we can chose a set of 18 local pure sliding moves
 consisting of 9 local pure sliding moves and their image by $\vec{\ }\circ {}^*$, or their image by ${}^{\times}$, which is for example the $9$ arrow in the left diagram in Figure \ref{Relations among MP moves2}.

\section{Cyclic moves}\label{cyclic}
In this section we study cyclic moves as we introduced in the introduction.
Note that the MP moves A2, A4, B2, B4, C3, D3 are cyclic, and the others are non-cyclic. 
For the pure sliding moves, ps-I and ps-II are non-cyclic, and ps-III and ps-IV are cyclic.
The inverse of a (non-) cyclic move will also be referred to as a (non) cyclic move.

\begin{prop}\label{cyclicprop}
Each non-cyclic MP move is derived as sequence of the primary MP move (which is non-cyclic), the non-cyclic pure sliding moves, and their inverses.
Any cyclic MP move cannot be derived as a sequence of non-cyclic moves.
Especially, each cyclic MP move must include a cyclic pure sliding moves in a sequence of the primary MP move, the pure sliding moves and their inverses representing it. 
\end{prop}

\begin{proof}
Refer to Figure \ref{cyclic diagram}, where we highlight the cyclic moves in the diagram from Figure \ref{Relations among MP moves}, depicting the relationships of the MP moves.
(Actually, the right diagram is unnecessary since it is the image of the involutions $\vec{\ }\circ {}^*$ or  ${}^{\times}$, as explained in the previous section, and these involution keep the cyclic property.)
\begin{figure}[H]
    \centering
    \includegraphics[scale=0.9]{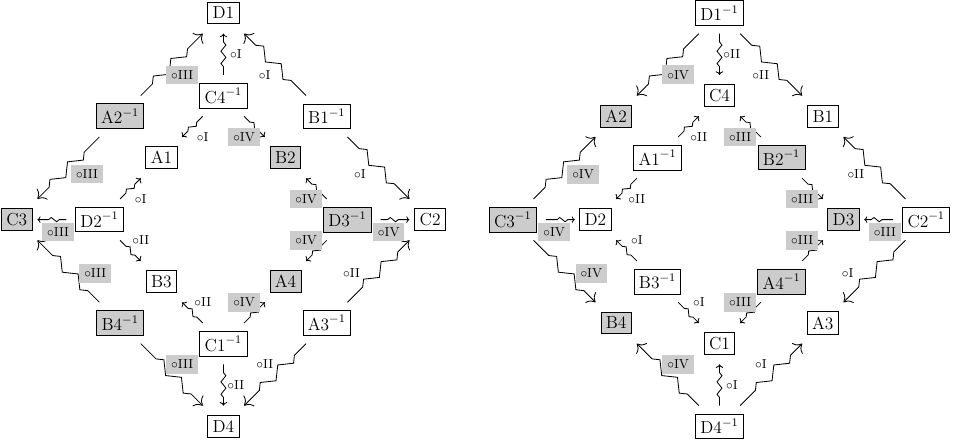}
    \caption{Cyclic MP moves and cyclic pure sliding moves in diagram in Figure \ref{Relations among MP moves}.}
    \label{cyclic diagram}
\end{figure}
Remember that the primary MP move is equivalent to the MP move D2. Figure \ref{cyclic diagram} illustrates that each non-cyclic MP move is a sequence of the primary MP move, the non-cyclic pure sliding moves, and their inverses. 

By using  the invariant $Z$ defined for integral normal o-graphs in \cite{MST2} for non-involutory Hopf algebras, we can show that any cyclic MP move cannot be derived as a sequence of non-cyclic moves. 
To be more precise, we apply $Z$ to a normal o-graph $\Gamma$ treating it as an integral normal o-graph $\Gamma(0)$ with a zero weight on each edge. To avoid confusion we denote the resulting map on normal o-graphs by $Z^0$, and we compare the value $Z^0(\Gamma)=Z(\Gamma(0))$ and $Z^0(\Gamma')=Z(\Gamma'(0))$ when $\Gamma$ and $\Gamma'$ are related by a (non-) cyclic move. Here, strictly speaking, we consider \textit{normal o-tangles} treating $Z$ as a functor from the category of integral normal o-tangles (cf. \cite[Remark 3.1]{MST2}), to compare the values of $Z^0$ locally around the move. 
While the invariant $Z^0$ remains unchanged under non-cyclic moves, it is not invariant under cyclic moves for non-involutory Hopf algebras. In fact, it is not difficult to observe that $Z^0$ is invariant under a cyclic move if and only if the Hopf algebra is involutory, see the proofs of \cite[Theorem 5.1]{MST1} and \cite[Theorem 3.3]{MST2}, see also Remark \ref{remcyclic} below. 
If a cyclic move were a sequence of non-cyclic moves, it would imply that the cyclic move does not change the value of $Z^0$, leading to a contradiction.
\end{proof}

We also classify integral moves as either cyclic or non-cyclic based on the underlying move for normal o-graphs.
\begin{cor}
Each integral non-cyclic MP move is derived as sequence of the integral primary MP move, the integral non-cyclic pure sliding moves, and their inverses.
Any integral cyclic MP move cannot be derived as a sequence of integral non-cyclic moves.
\end{cor}
\begin{proof}
The proof for the former part is similar to that of the former part of Proposition \ref{cyclicprop}, once integer weights are attached. 
For the latter part, we can  extend the arguments for the latter part of Proposition \ref{cyclicprop} to integral normal o-graphs and integral moves  by considering the underlying normal o-graphs.
\end{proof}

\begin{rem}\label{remcyclic}
In \cite{MST1}, we introduced an invariant $Z$ of combed $3$-manifolds using normal o-graphs and involutory Hopf algebras, which is essentially equivalent to $Z^0$ discussed in the proof of Proposition \ref{cyclicprop}. To avoid confusion with $Z$ introduced in \cite{MST2}, we will denote the former as $Z^0$ in the following discussion.
In \cite{MST2}, we extended $Z^0$ to integral normal o-graphs and any Hopf algebra to define $Z$. In this extended formulation, the integer weight $i$ on an edge of an integral normal o-graph affects as $S^{2i}$ with the antipode $S$ of the Hopf algebra. Thus, when we work with an involutory Hopf algebra, the invariant $Z$ for integral normal o-graphs reduces to $Z^0$ for normal o-graphs by applying it to underlying normal o-graphs, resulting in an invariant of combed $3$-manifolds. If the involutory Hopf algebra is in addition unimodular and counimodular, then $Z$ becomes invariant under the CP move and gives a topological invariant of $3$-manifolds.

For non-involutory Hopf algebras, $Z^0$ fails to be an invariant for combed $3$-manifolds because it is not invariant under cyclic moves, although it remains invariant under non-cyclic moves. To address this, in \cite{MST2}, we introduced integral normal o-graphs and extended $Z$ from $Z^0$ ensuring its invariance under integral cyclic moves. This extension results in an invariant for framed $3$-manifolds. 

Note that integral cyclic moves have non-trivial integer weights as illustrated in Figure \ref{fig:MP-move}, which help reconcile the disparities  of the values of $Z^0$ before and after cyclic moves. On the other hand, integral non-cyclic moves have everywhere zero weights and thus $Z$ is compatible with $Z^0$, i.e., $Z((\Gamma,x))=Z((\Gamma,0))=Z^0(\Gamma)$ locally around non-cyclic moves, tearing $Z$ as a functor from  integral normal o-tangles as mentioned in the proof of Proposition \ref{cyclicprop}. 

In the above argument, the obstruction preventing $Z^0$ for a non-involutory Hopf algebra from becoming an invariant of combed $3$-manifolds (or topological $3$-manifolds) is the invariance under cyclic moves.  From the viewpoint of Proposition \ref{cyclicprop}, the obstruction lies essentially in the cyclic pure sliding moves, specifically ps-III and ps-IV. We can strengthen this statement further. 
As depicted in Figure \ref{cyclic diagram}, each cyclic MP move is a sequence consisting of a non-cyclic MP move  (or its inverse) and a cyclic pure sliding move (or its inverse).
Consequently, $Z^0$ is invariant under a cyclic MP move if and only if $Z^0$ is invariant under cyclic pure sliding moves.
A straight forward calculation shows that $Z^0$ is invariant under ps-III (or equivalently ps-IV) if and only if the Hopf algebra is involutory. 
To summarize, $Z^0$ is invariant under cyclic moves if and only if the Hopf algebra is involutory. 

In other words, the invariant $Z$ for a non-involutory Hopf algebra captures framings of $3$-manifolds (at least locally; if not, the Hopf algebra should be involutory), and algebraic equations in the non-involutory Hopf algebra coming from the geometrical equivalence relation (homotopy) of framings via $Z$ are reduced to algebraic properties of the square of the antipode. Specifically, if $S^2=1$, the equations become trivial.  Further study of the relationship between framings  and the properties of $S^2$ would be interesting.
\end{rem}

\end{appendices}

\end{document}